\documentclass[reqno,12pt]{amsart}
\usepackage{amscd,amsfonts,mathrsfs,amsthm,amssymb, amsmath,mathtools}
\usepackage{upgreek}
\usepackage{csquotes}
\usepackage{enumerate}
\usepackage{bbm}
\usepackage{multicol}
\usepackage{color}
\usepackage{array}
\usepackage{ae}
\usepackage{accents}
\usepackage{epsfig}
\usepackage[e]{esvect}
\usepackage[nobysame]{amsrefs}
\usepackage[margin=1in]{geometry}

\let\oldtocsection=\tocsection
\let\oldtocsubsection=\tocsubsection
\let\oldtocsubsubsection=\tocsubsubsection
\renewcommand{\tocsection}[2]{\hspace{0em}\oldtocsection{#1}{#2}}
\renewcommand{\tocsubsection}[2]{\hspace{1em}\oldtocsubsection{#1}{#2}}
\renewcommand{\tocsubsubsection}[2]{\hspace{2em}\oldtocsubsubsection{#1}{#2}}

\def\real     #1{{\mathbb R^{#1}}}
\def\complex  #1{{\mathbb C^{#1}}}
\def\natural  #1{{\mathbb N^{#1}}}

\def\uu       #1#2#3{{#1}^{#2#3}}

\def\ud       #1#2#3{{#1}^{#2}_{\phantom{#2}{#3}}}

\def\ds       {{\text{\normalsize$\rm\frac{d}{ds}$}}}
\def\dt       {{\text{\normalsize$\rm\frac{d}{dt}$}}}

\def\equationcolor {\color{black}}
\def\textcolor     {\color{black}}

\def\bcoleq    {\begin{equation}\equationcolor}
\def\ecoleq    {\textcolor\end{equation}}
\def\bcoleqn   {\equationcolor\begin{eqnarray}}
\def\ecoleqn   {\end{eqnarray}\textcolor}

\def\sind{\operatorname{S}}
\def\tind{\operatorname{T}}
\def\pind{\operatorname{p}}

\def\gind{\operatorname{g}}

\def\F{\mathbf{F}}
\def\f{\mathbf{f}}
\def\a{\mathbf{a}}
\def\b{\mathbf{b}}

\def\dn{\mathrm{d}}
\def\x{\mathbf{x}}
\def\y{\mathbf{y}}
\def\H{\mathbf{H}}
\def\D{\mathrm{D}}
\def\M{\mathrm{M}}

\def\tmax{\mathrm{T}}

\def\gind{\operatorname{g}}
\def\pind{\operatorname{p}}
\def\qind{\operatorname{q}}
\def\sind{\operatorname{S}}
\def\tind{\operatorname{T}}
\def\varmine{{\upphi}}

\newtheorem{theorem}{Theorem}[section]

\newtheorem{lemma}[theorem]{Lemma}

\newtheorem{corollary}[theorem]{Corollary}

\newtheorem{proposition}[theorem]{Proposition}
\newtheorem{definition}[theorem]{Definition}
\theoremstyle{definition}
\newtheorem{remark}[theorem]{Remark}
\newtheorem{example}[theorem]{Example}
\numberwithin{equation}{section}

\usepackage{accents}
\newcommand*{\bdot}[1]{%
	\accentset{\mbox{\small\bfseries .}}{#1}}

\newcommand{\scaldot}{\,\begin{picture}(-1,1)(-1,-3)\circle*{3}\end{picture}\hspace{1.5pt}\ }

\begin{document}

\title[Codimension two mean curvature flow of entire graphs]{Codimension two mean curvature flow of entire graphs}

\author[Andreas Savas-Halilaj]{\textsc{Andreas Savas-Halilaj}}
\address{%
         Andreas Savas-Halilaj\newline
	University of Ioannina\newline
	Department of Mathematics\newline
	Section of Algebra and Geometry\newline
	45110 Ioannina, Greece}
\email{ansavas@uoi.gr}

\author[Knut Smoczyk]{\textsc{Knut Smoczyk}}
\address{%
	Knut Smoczyk\newline
	Leibniz University Hannover\newline
	Institute of Differential Geometry\newline
	and Riemann Center for Geometry and Physics\newline
	Welfengarten 1\newline
	30167 Hannover, Germany}
	\email{smoczyk@math.uni-hannover.de}

\date{}


\begin{abstract}
	We consider the graphical mean curvature flow of maps $\f:\real{m}\to\real{n}$, $m\ge 2$, and derive estimates on the growth rates of the evolved graphs, based on a new version of the maximum principle for properly immersed submanifolds that extends the well-known maximum principle of \textsc{Ecker} and \textsc{Huisken} derived in their seminal paper \cite{EH}. In the case of uniformly area decreasing maps $\f:\real{m}
\to\real{2}$, $m\ge 2$, we use this maximum principle to show that the graphicality and  the area decreasing property are preserved. Moreover, if the initial graph is asymptotically conical at infinity, we prove that the normalized mean curvature flow smoothly converges to a self-expander.
\end{abstract}
\maketitle

\section*{Introduction and summary}
The mean curvature flow of entire graphs of codimension one in the
{euclidean} space was introduced by \textsc{Ecker} and \textsc{Huisken} in 
their seminal paper \cite{EH}. They proved that the graphical property and 
polynomial growth conditions for the graphical function are preserved under 
the flow. Moreover, if the initial graph satisfies a linear growth condition and is 
asymptotic to a cone at infinity, then one obtains a longtime solution and in 
this case the blow-down of the flow converges to a self-expanding entire 
graph.

In this paper, we consider immersions 
$$\F:\M^m\to\real{m+n}$$
of $m$-dimensional submanifolds in $\real{m+n}$. We say that $\M^m$ moves by mean curvature
flow if there is a one-parameter family $\F_t=\F(\cdot,t)$ of immersions with corresponding images $\M_t=\F_t(\M^m)$ such that
\begin{gather}
\dt \F(p,t)=\H(p,t),\quad \F(p,0)=\F_0(p),\quad p\in \M,\tag{MCF}\label{mcf}
\end{gather}
is satisfied for some initial data $\F_0$. Here $\H(p,t)$ is the mean curvature vector of the submanifold $\M_t$ at $\F(p,t)$.

Throughout this paper we shall assume that $\M_0$ can be written as an entire graph; i.e.
$\M^m=\real{m}$, and there exists a map $\f_0:\real{m}\to\real{n}$ such that
$$\F_0(\x)=(\x,\f_0(\x)), \quad\x\in\real{m}.$$
If the evolving submanifolds stay graphical, then up to tangential diffeomorphisms \eqref{mcf} is equivalent to the quasilinear system
\begin{gather}
\dt\f(\x,t)=\operatorname{tr}_{\gind}\D^2\f,\quad\x\in\real{m},\tag{MCF'}\label{mcf nonp}
\end{gather}
where $\D^2\f$ is the usual \textsc{Hessian} in $\real{m}$ and $\operatorname{tr}_{\gind} $ denotes 
the trace with 
respect to the induced time-dependent metric $\gind$ on the evolving graph. In particular, 
\eqref{mcf nonp} becomes uniformly parabolic if  $\vert \D \f\vert^2$ is uniformly bounded.

	The objective of this work is to investigate the behavior of the mean curvature flow of complete entire graphs in euclidean space in higher codimension. In the first part, based on a refined version of the parabolic maximum principle for complete and properly immersed solutions that we derive in Section \ref{sec1}, we show that various growth conditions of the initial graph are preserved under the  mean curvature flow in any codimension. 
	
In the second part of the paper we apply these results to  uniformly area decreasing maps in 
codimension two, and prove that this property is preserved. In Sections \ref{sec5} and 
\ref{sec6} we address the longtime existence of the flow for  uniformly area decreasing maps 
in codimension two, and show that our family of graphs asymptotically approaches a self-
expander if the initial graph $\F_0$ is asymptotically conical at infinity, i.e. if there exist 
constants $\delta,c>0$ such that
$$
|\F^{\perp}_0|\le c\big(1+|\F_0|\big)^{1-\delta}.
$$
	More precisely, we prove that under the normalized mean curvature flow, the evolved graphs
$\F_t:\real{m}\to\real{m+2}$, $t\ge 0$, converge to a graph $\F_{\infty}:\real{m}\to
\real{m+2}$ satisfying
$$
\F_{\infty}=\H^{\perp}_{\infty},
$$
	where $\{\cdot\}^\perp$ denotes the orthogonal projection to the normal bundle of the submanifold. This result extends \textsc{Ecker} and \textsc{Huisken}'s theorem in the case of entire graphs in codimension two and is sharp. However, the proof significantly differs and is based on a priori estimates for the mean curvature of the evolved submanifolds, some of which are interior in nature.

	Finally, in the appendix, we demonstrate how our method can be used to obtain a \textsc{Bernstein} type result for entire area decreasing graphs of codimension two.

\section{Maximum principles for entire graphs}\label{sec1}
For a fixed point $(\y_0,t)\in\real{m+n}\times\real{}$ let the backward heat kernel $k=k(\y,t)$ in $\real{m+n}$ centered at $(\y_0,t)$ be  defined by $$k(\y,t)=(4\pi\tau)^{-(m+n)/2}\exp\left(\frac{-\vert\y_0-\y\vert^2}{4\tau}\right),\quad\tau:=t_0-t, \quad t_0>t,$$
and let
$$\rho(\y,t):=(4\pi\tau)^{n/2}k(\y,t)=(4\pi\tau)^{-m/2}\exp\left(\frac{-\vert\y_0-\y\vert^2}{4\tau}\right),\quad\tau:=t_0-t, \quad t_0>t.$$
It was shown by {\sc Huisken} \cite{Huisken90}, and later generalized by {\sc Hamilton} \cite{Hamilton93} to arbitrary co- dimension that this gives the monotonicity formula
\begin{equation*}
\dt\int_{\M_t}\rho\, d\mu_t=-\int_{\M_t}\rho\left\vert \H+\frac{1}{2\tau}(\F-\y_0)^\perp\right\vert^2d\mu_t,
\end{equation*}
where $d\mu_t$ is the measure on $\M_t$. More generally for a function $u=u(p,t)$ on a hypersurface $\M$, {\sc Huisken} \cite{Huisken90} derived the formula 
\begin{equation}\label{mon 2}
\dt\int_{\M_t}u\rho\, d\mu_t=\int_{\M_t}\left(\dt u-\Delta u\right)\rho\,d\mu_t-\int_{\M_t}u\rho\left\vert \H+\frac{1}{2\tau}(\F-\y_0)^\perp\right\vert^2d\mu_t,
\end{equation}
which carries over unchanged to any codimension. To guarantee that the integrals are finite
we assume throughout this paper that all the functions that we are going to consider on
$\M^m$ satisfy a suitable sub-exponential growth and integration by parts is permitted; for more details we refer to \cite[Chapter 4]{ecker}.

Since the manifold $\M^m$ is not compact, the usual parabolic maximum principles cannot be applied but based on the monotonicity formula, {\sc Ecker} and {\sc Huisken} \cite{EH} proved the following maximum principle.
\begin{proposition}[{\sc Ecker} and {\sc Huisken} \cite{EH}]\label{prop 1}
	Suppose the function $u=u(p,t)$ satisfies the inequality
	\begin{equation*}\label{est 1}
	\left(\dt-\Delta\right)u\le\langle\a,\nabla u\rangle
	\end{equation*}
	for some vector field $\a$, where $\nabla$ denotes the tangential gradient on $\M^m$.
	If 
	$$a_0=\sup_{\M^m\times[0,t_1]}\vert\a\vert<\infty$$
	for some $t_1>0$, then
	$$\sup_{\M_t}u\le \sup_{\M_0}u$$
	for all $t\in[0,t_1]$.
\end{proposition}
This maximum principle is very powerful and served as one of the key techniques in \cite{EH} but it cannot be applied in many cases because of the required special form of the evolution inequality. We will now derive more general versions of this principle.
\begin{proposition}\label{prop 2}
	Suppose the function $u=u(p,t)$ satisfies the inequality
	\begin{equation*}\label{est 2}
	\left(\dt-\Delta\right)u\le \langle\a,\nabla u\rangle+\b u
	\end{equation*}
	for a vector field $\a$ and a function $\b$, where $\a$, $\b$ are uniformly bounded on
	$\M^m\times[0,t_1]$, for some $t_1>0$. Then
	$${\sup}_{\M_t}u\le{\sup}_{\M_0}u \cdot e^{\beta t},\, t\in[0,t_1],$$
	where
	$$\beta:=\begin{cases}\sup_{\M\times[0,t_1]}\b&\text{, if }\,\,\sup_{\M_0}u >0,\\
	\inf_{\M\times[0,t_1]}\b&\text{, if }\,\,\sup_{\M_0}u <0,\\
	0&\text{, if }\,\,\sup_{\M_0}u =0.
	\end{cases}$$
\end{proposition}
\begin{proof} The function $\mu:[0,t_1]\to\mathbb{R}$ given by
	$\mu(t)=\sup_{\M_0}u \cdot e^{\beta t}$
	is the solution of
	$$
\bdot\mu(t)=\beta\mu(t),\quad \mu(0)={\sup}_{\M_0}u.
	$$
	Note that
	\begin{eqnarray*}
		\left(\dt-\Delta\right)(u-\mu)&\le&  \langle\a,\nabla (u-\mu)\rangle+\b (u-\mu)+(\b-\beta)\mu.
	\end{eqnarray*}
	By our choice of $\beta$ we get always
	$$(\b-\beta)\mu\le 0,\quad \text{for all}\quad t\in[0,t_1].$$
	Let us consider now the locally  \textsc{Lipschitz} continuous function
	$u_\mu:\M^m\times[0,t_1]\to\mathbb{R}$ given by
	$u_\mu:=\max\{u-\mu,0\}.$
	We claim that $u_\mu\equiv 0$. Indeed, the function $u_{\mu}$
	satisfies
	\begin{eqnarray*}
		\left(\dt-\Delta\right)u_\mu&\le&  \langle\a,\nabla u_\mu\rangle+\delta u_\mu,
	\end{eqnarray*}
	for all $t\in[0,t_1]$ and $\delta=\sup_{\M\times[0,t_1]}\b$. Using the boundedness of $\a$ and \textsc{Young}'s inequality we get
	\begin{eqnarray*}
		\left(\dt-\Delta\right)u_\mu^2&\le&  2u_\mu\langle\a,\nabla u_\mu\rangle-2\vert\nabla u_\mu\vert^2+2\delta u_\mu^2\le cu_\mu^2,
	\end{eqnarray*}
	for some non-negative constant $c$ and for all $t\in[0,t_{max}]$. We may now employ \eqref{mon 2} with $u_\mu^2$ instead of $u$ and $t_0=t_1$, $\y_0$ arbitrary in the definition of $\rho$ to conclude
	$$\dt\int_{\M_t}u_\mu^2\rho\,d\mu_t\le c\int_{\M_t}u_\mu^2\rho\,d\mu_t
	\quad\text{and}\quad \int_{\M_0}u_\mu^2\rho\,d\mu_0=0,$$
	which gives $u_\mu\equiv 0$, for all $t\in[0,t_1]$. This finishes the proof.
\end{proof}
\begin{proposition}
	Let $W\subset\real{}$ be a closed domain and $\Phi:W\to\real{}$ a uniformly \textsc{Lipschitz} continuous function.
	Suppose the function $u=u(p,t)$ takes values in $W$ for all $(p,t)\in \M^m\times[0,t_1]$,
	for some $t_1>0$, and that $u$
	satisfies the inequality
	\begin{equation*}\label{est 3}
	\left(\dt-\Delta\right)u\le \langle\a,\nabla u\rangle+\Phi(u),
	\end{equation*}
	where the vector field $\a$ is uniformly bounded. If $\sup_{\M_0}u<\infty$, and
	$k:[0,t_0]\to\real{}$ is the solution of the ODE
	$$
	\begin{cases}
	\bdot k(t)=\Phi(k(t)),\\
	k(0)=\mu_0:=\sup_{\M_0}u,
	\end{cases}$$
	for some $t_0\in(0, t_1]$, then
	$${\sup}_{\M_t}u\le k(t), \,\text { for all } t\in[0,t_0].$$
\end{proposition}
\begin{proof}
	The function $u-k$ satisfies
	$$\left(\dt-\Delta\right)(u-k)\le \langle\a,\nabla (u-k)\rangle+\Phi(u)-\Phi(k).$$
	Since $\Phi$ is uniformly \textsc{Lipschitz} on $W$, there exists a constant $L\ge 0$ such that
	$$\vert\Phi(u)-\Phi(k)\vert\le L\vert u-k\vert,\,\text{  for all }u,k\in W.$$
	Therefore
	$$\left(\dt-\Delta\right)(u-k)\le \langle\a,\nabla (u-k)\rangle+L\vert u-k\vert.$$
	Similarly as in the proof of Proposition \ref{prop 2} define the function
	$u_k:=\max\{u-k,0\}.$
	Then, in the weak sense, we get
	$$\left(\dt-\Delta\right)u_k\le \langle\a,\nabla u_k\rangle+Lu_k,$$
	for all $t\in[0,t_0]$. This implies
	$$\left(\dt-\Delta\right)u_k^2\le cu_k^2,$$
	for some constant $c$ and we may proceed as in the proof of Proposition \ref{prop 2} to conclude that $u_k\equiv 0$, for all $t\in[0,t_0]$.
\end{proof}

\section{Short time existence}\label{sec2}
Consider a smooth map $\f:\real{m}\to \real{n}$, fix a point $\x\in \real{m}$ and let $\lambda^2_{1}\ge\displaystyle{\dots}\ge\lambda^2_{m}$ denote the eigenvalues of $\f^{*}\langle\cdot,\cdot\rangle_{\real{n}}$ at $\x$ with respect to $\langle\cdot,\cdot\rangle_{\real{m}}$. The corresponding
values $\lambda_i\ge  0$, $i\in\{1,\dots,m\}$, are the {\em singular values} of the differential $\D\f$ of $\f$ at the point $\x$. The singular values are \textsc{Lipschitz} continuous functions on $\real{m}$. 

\begin{proposition}[Short time existence]\label{ste}
	Let $\M_0$ be the graph of a smooth map $\f:\real{m}\to\real{n}$ with $\sup_{\M_0}\vert\D\f\vert$, $\sup_{\M_0}\vert\nabla^{\ell} A\vert$ bounded for all $\ell\ge 0$, where $A$ denotes the second fundamental form of the graph. Then there exists a unique smooth solution on some short time interval $[0,t_0)$, $t_0>0$, with 
	$${\sup}_{\M_t}\vert\D\f\vert<\infty\quad\text{and}\quad {\sup}_{\M_t}\vert\nabla^{\ell} A\vert<\infty,$$
	for all $\ell\ge 0$ and $t\in[0,t_0)$.
\end{proposition}
\begin{proof}
	The proof follows similarly as in \cite[Theorem 4.6]{EH}; see also \cite[Proposition 5.1]{CCH12}) and
\cite{Hamilton82}. Since all singular values of the initial map $\f$ are uniformly bounded, the linearization\footnote{The coefficients of the symbol are given by the inverse of the \textsc{Riemannian} metric $\gind$ on $\M_0$.} of the quasilinear system \eqref{mcf nonp} becomes uniformly parabolic on $\real{m}$, and a unique smooth solution with
$${\sup}_{\M_t}\vert\D\f\vert<\infty\quad\text{and}\quad{\sup}_{\M_t}\vert\nabla^\ell A\vert<\infty,
\quad\ell\ge 0,
$$
is guaranteed on some short time interval by the implicit function theorem. 
\end{proof}
	It should be noted here that, as in the case for the standard heat equation,  the uniqueness of the solution to the initial value problem in Proposition \ref{ste} is guaranteed only within the class of maps that satisfy
	$${\sup}_{\M_t}\vert\D\f\vert<\infty\quad\text{and}\quad{\sup}_{\M_t}\vert\nabla^\ell A\vert<\infty,
	\quad\ell\ge 0.
	$$
	In general, for complete and non-compact initial data, the existence and uniqueness of  solutions of \eqref{mcf} is not automatic;
	for more details see
	\cite{DS23}, \cite{ES}, \cite{Giga} and  \cite{Saez}. Moreover,  rather recently there have been some interesting works addressing the longtime behavior of the mean curvature flow of graphical submanifolds in euclidean spaces under the mean curvature flow; see for example \cite{CCY13}, \cite{CCH12}, \cite{DJX23}, \cite{Lubbe1} and \cite{Lubbe2}.

In the sequel we assume that $\M_t$, $t\in[0,\tmax)$, is a smooth solution of the mean curvature flow, such that $\M_0$ is the graph of a smooth map $\f:\real{m}\to\real{n}$ with
\begin{gather}\label{initial}
{\sup}_{\M_0}\vert\nabla^\ell A\vert<\infty,\text { for all }\,\ell\ge 0,\tag{$\ast$}
\end{gather}
where
$$\tmax:=\sup\bigl\{t_0\ge 0:\text{a smooth solution of \eqref{mcf} exists on $[0,t_0)$}\bigr\}$$
denotes the maximal time of existence.
To obtain longtime existence results, it is crucial to derive a priori estimates for the curvature. Let us introduce a few tensors in terms of the second fundamental form $A$ and its mean curvature $\H$. We define
\begin{eqnarray*}
	A^{\H}(v,w)&:=&\langle\H, A(v,w)\rangle,\\
	(A\scaldot A)(v_1,v_2,w_1,w_2)&:=&\langle A(v_1,v_2),A(w_1,w_2)\rangle,
\end{eqnarray*}
where all vectors are tangent vectors. The evolution equations of the mean curvature $\H$ and of the second fundamental form $A$ are well known; see for example \cite{Smoczyk12}.
\begin{lemma}\label{lemma 2.2}
	The evolution equations of $\,\H$, $\vert \H\vert^2$, and $\vert A\vert^2$ are given by
	\begin{eqnarray*}
	\Bigl(\nabla^\perp_\dt-\Delta^\perp\Bigr)\H&=&\langle A^{\H},A\rangle,\label{mcf mean}\\
	\Bigl(\dt-\Delta\Bigr)\vert \H\vert^2&=&-2\vert\nabla^\perp\H\vert^2+2\vert A^{\H}\vert^2,\label{mcf 5}\\
	\Bigl(\dt-\Delta\Bigr)\vert A\vert^2&=&-2\vert\nabla^\perp A\vert^2+2\vert A\scaldot A\vert^2+2\vert \text{\rm R}^\perp\vert^2\label{mcf 8},
	\end{eqnarray*}
	where $\nabla^\perp$ denotes the connection of the normal bundle, and $\text{\,\rm R}^\perp$ is its curvature tensor.
\end{lemma}
\begin{remark}
	In the sequel it will be necessary to have a good estimate for the curvature terms on the right hand side of the evolution equation for $\vert A\vert^2$. By a well known result (see \cite[Theorem 1]{LiLi92}) one always has
	\begin{equation}\label{secest}
	2\vert A\scaldot A\vert^2+2\vert \text{\rm R}^\perp\vert^2\le 3\vert A\vert^4.
	\end{equation}
\end{remark}

Using the evolution equations of $A$ and of the Levi-Civita connection $\nabla$, one can proceed similarly as in \cite[Theorem 7.1]{Huisken84} to prove the next  lemma.
\begin{lemma}\label{lemma 2.4}
	For arbitrary $\ell\ge 0$ there exists a constant $C(\ell,m)$, depending only on $\ell$ and $m$, such that
	\begin{equation*}\label{mcf 10}
	\left(\dt-\Delta\right)\vert\nabla^{\ell}A\vert^2\le-2\vert\nabla^{\ell+1}A\vert^2+C(\ell,m)\sum_{i+j+k=\ell}\vert\nabla^iA\vert\vert\nabla^jA\vert\vert\nabla^kA\vert\vert\nabla^{\ell}A\vert.
	\end{equation*}
\end{lemma}
Using the above lemma, we can easily derive a priori $C^k$-estimates in terms of $C^2$-estimates.
\begin{proposition}\label{longtime}
	Let $\M_t$, $t\in[0,T)$,  be a smooth solution of \eqref{mcf} with bounded curvature on each $\M_t$, and suppose $\M_0$ is the graph of a smooth map $\f:\real{m}\to\real{n}$ such that $\sup_{\M_0}\vert\D\f\vert$, $\sup_{\M_0}\vert\nabla^{\ell} A\vert$ are bounded for all $\ell\ge 0$. Then for $t\ge 0$ the submanifold $\M_t$ is a graph of a $C^1$-bounded map, and for all $t_0<T$, $\ell\ge 1$ we have the a priori estimate
	$$ {\sup}_{\M_t}\vert\nabla^\ell A\vert\le C(l,t_0),$$
	where $C(l,t_0)$ depend only on $\ell$, $m$, $\sup_{\M^m\times[0,t_0]}\vert A\vert$ and $\sup_{\M_0}\vert\nabla^j A\vert$, $1\le j\le \ell$. In particular $\vert\nabla^\ell A\vert$ becomes uniformly bounded in time once this holds for $\vert A\vert$.
\end{proposition}
\begin{proof}
	Using Lemma \ref{lemma 2.4} one may proceed exactly as in \cite[Proposition 2.3]{Huisken90} to obtain the a priori estimates for $\vert\nabla^\ell A\vert$. To show that $\M_t$ is a graph of a $C^1$-bounded map one proceeds as follows. From the evolution equation of $\F$ one concludes that $\D\f$ satisfies on each finite time interval $[0,t_0]$, $t_0<T$, an estimate
	$$\left(\dt-\Delta\right)\vert\D\f\vert^2\le C\vert\D\f\vert^2,$$
	where $C$ depends only on $\sup_{\M_0}\vert\D\f\vert$, $\sup_{\M^m\times[0,t_0]}\vert A\vert$. Then Proposition \ref{prop 1} implies that $\sup_{\M_t}\vert\D\f\vert$ can grow at most exponentially in time on $[0,t_0]$.
\end{proof}
\begin{remark}
	Proposition \ref{longtime} implies that longtime existence of the graphical mean curvature flow follows from uniform $C^2$-estimates. Note that uniform estimates of $\vert\D\f\vert$ are usually needed to derive such uniform $C^2$-estimates, but that uniform $C^1$-estimates do not follow from Proposition \ref{longtime} if $\vert A\vert $ stays uniformly bounded in space and time. In addition, in contrast to the compact case, bounded curvature on each $\M_t$ is not guaranteed for all times $t\in[0,T)$. 
\end{remark}

\section{Growth rates in arbitrary codimension}\label{sec3}
In this section we derive estimates on growth rates that do not even require $\M_t$ to be a graph. To express growth rates of an immersion $\F:\real{m}\to\real{m+n}$ let us consider the natural projections
$$\pi_{\x}:\real{m}\times\real{n}\to\real{m},\quad \pi_{\y}:\real{m}\times\real{n}\to\real{n},$$ 
and define
$$\x:=\pi_{\x}\circ\F,\quad \f:=\pi_{\y}\circ \F.$$
Since
$$\Delta\F=\H$$
where $\Delta$ is the \textsc{Laplace-Beltrami} operator on $\M^m$, one can rewrite \eqref{mcf} as the heat equation (where $\Delta$ is depending on $\F$), i.e.
\begin{equation}\label{mcf 1}
\left(\dt-\Delta\right) \F=0,\quad 
\left(\dt-\Delta\right) \x=0,\quad \left(\dt-\Delta\right)\f=0.
\end{equation}
Similarly as in \cite{EH} we define
\begin{equation}\label{defeta}
\upeta:=\vert\x\vert^2+2mt+\uplambda(t),
\end{equation}
where $\uplambda$ will be chosen later.
\begin{lemma}\label{lemma 3.1}
	Let $\varmine:\real{}\to\real{}$ be an arbitrary smooth function. Then  $\varmine=\varmine(\upeta)$ satisfies 
	$$
	\left(\dt-\Delta\right)\varmine=\bigl(2\vert\nabla\f\vert^2+\bdot\uplambda\bigr)\varmine'-\vert\nabla\upeta\vert^2\varmine''.
	$$
\end{lemma}
\begin{proof}
Since
$$\upeta=\vert\F\vert^2-\vert\f\vert^2+2mt+\uplambda\quad\text{and}\quad
\vert\nabla\F\vert^2=m,
$$
we may use \eqref{mcf 1} to conclude the equation for $\upeta$. The evolution equation for $\varmine$ then follows from the chain rule.
\end{proof}
We are now ready to prove that polynomial growth rates are preserved.
\begin{proposition}\label{prop 3}
	Let $\varmine:[0,\infty)\to\real{}$ be a smooth function such that
	\begin{equation}
	\varmine>0,\quad \varmine'\ge 0,\quad
	({s^{-k}\varmine'})'\le 0
	,\,\text{  for    } s>0 \text{  and some } k\ge -m/2.\nonumber
	\end{equation}
	If the inequality
	$$\vert\f\vert^2\le \varmine(\vert\x\vert^2)$$
	is satisfied on $\M_0$, then for all $t\ge 0$, it holds
	\begin{equation}
	\vert\f\vert^2\le \varmine(\vert\x\vert^2+2(m+2k)t).\nonumber
	\end{equation}
\end{proposition}
\begin{proof}
It suffices to prove that, for any $t\ge 0$, it holds
$$\sup_{\M_t}\frac{\vert\f\vert^2}{\varmine(\upeta)}\le 1,\text{ where}\quad \upeta=\vert\x\vert^2+2mt+4kt.$$
So we may employ Lemma \ref{lemma 3.1} with $\uplambda(t):=4kt$, and compute
	\begin{eqnarray}
	\left(\dt-\Delta\right) \frac{\vert\f\vert^2}{\varmine}&=&\frac{1}{\varmine}\left(\dt-\Delta\right)\vert\f\vert^2-\frac{\vert\f\vert^2}{\varmine^2}\left(\dt-\Delta\right)\varmine\nonumber\\
	&&-2\frac{\vert\f\vert^2}{\varmine^3}\vert\nabla\varmine\vert^2+\frac{2}{\varmine^2}\left\langle\nabla\vert\f\vert^2,\nabla\varmine\right\rangle\nonumber\\
	&=&-\,\frac{2}{\varmine}\vert\nabla\f\vert^2-2\frac{\varmine'}{\varmine^2}\vert\f\vert^2\vert\nabla\f\vert^2-4k\frac{\varmine'}{\varmine^2}\vert\f\vert^2\nonumber\\
	&&+\frac{\varmine''}{\varmine^2}\vert\f\vert^2\vert\nabla\upeta\vert^2-2\frac{(\varmine')^2}{\varmine^3}\vert\f\vert^2\vert\nabla\upeta\vert^2+2\frac{\varmine'}{\varmine^2}\left\langle\nabla\vert\f\vert^2,\nabla\upeta\right\rangle\nonumber.
	\end{eqnarray}
	Hence
	\begin{eqnarray}\label{prelim 1}
	\left(\dt-\Delta\right) \frac{\vert\f\vert^2}{\varmine}&=&-\,\frac{2}{\varmine}\Bigl(1+\frac{\varmine'}{\varmine}\vert\f\vert^2\Bigr)\vert\nabla\f\vert^2+\Bigl(\frac{\varmine''}{\varmine^2}-2\frac{(\varmine')^2}{\varmine^3}\Bigr)\vert\f\vert^2\vert\nabla\upeta\vert^2\\
	&&-4k\frac{\varmine'}{\varmine^2}\vert\f\vert^2+2\frac{\varmine'}{\varmine^2}\left\langle\nabla\vert\f\vert^2,\nabla\upeta\right\rangle\nonumber.
	\end{eqnarray}
	Denoting by $\{v_i\}_{i=1,\dots,m+n}$ the standard orthonormal basis for $\real{m+n}$, we have
	$$\f=\sum_{\alpha>m}\langle\F,v_\alpha\rangle v_\alpha,$$
	so that with $\f^\alpha:=\langle\F,v_\alpha\rangle$ we get
	$$\vert \f\vert^2=\sum_{\alpha>m}(\f^\alpha)^2,\quad\vert\nabla\f\vert^2=\sum_{\alpha>m}\vert\nabla \f^\alpha\vert^2.$$
Using the \textsc{Peter-Paul} inequality, we obtain for any $a>0$
	\begin{eqnarray}
	\vert \langle\nabla\vert\f\vert^2,\nabla\upeta\rangle\vert&=&2\Bigl\vert \sum_{\alpha>m}\f^\alpha\langle\nabla\f^\alpha,\nabla\upeta\rangle\Bigr\vert\nonumber\\
	\phantom{\left(\dt-\Delta\right)}&\le&\sum\limits_{\alpha> m}\left(a\vert\nabla\f^\alpha\vert^2+a^{-1}(\f^\alpha)^2\vert\nabla\upeta\vert^2\right)\nonumber\\
	\phantom{\left(\dt-\Delta\right)}&=&a\vert\nabla\f\vert^2+a^{-1}\vert\f\vert^2\vert\nabla\upeta\vert^2.\nonumber
	\end{eqnarray}
	Combining this with \eqref{prelim 1}  gives
	\begin{eqnarray}
	\left(\dt-\Delta\right) \frac{\vert\f\vert^2}{\varmine}&\le&-\,\frac{2}{\varmine}\Bigl(1+\frac{\varmine'}{\varmine}\vert\f\vert^2-a\frac{\varmine'}{\varmine}\Bigr)\vert\nabla\f\vert^2\nonumber\\
	&&+\Bigl(\frac{\varmine''}{\varmine^2}-2\frac{(\varmine')^2}{\varmine^3}+2\frac{\varmine'}{\varmine^2}a^{-1}\Bigr)\vert\f\vert^2\vert\nabla\upeta\vert^2-4k\frac{\varmine'}{\varmine^2}\vert\f\vert^2.\nonumber
	\end{eqnarray}
	The conditions for $\varmine$ imply $\varmine''\upeta\le k\varmine'$. Therefore, at points where $\varmine'=0$, the above estimate implies
	$$\left(\dt-\Delta\right) \vert\f\vert^2\varmine^{-1}\le0.$$
	At all other points we choose $a:={\varmine}/{\varmine'}$ to derive 
	\begin{eqnarray}
	\left(\dt-\Delta\right)\frac{\vert\f\vert^2}{\varmine}&\le&\left(\varmine''\vert\nabla\upeta\vert^2-4k\varmine'\right)\frac{\vert\f\vert^2}{\varmine^2}.\nonumber
	\end{eqnarray}
	Since $\nabla\upeta=2\x^\top$, where $\x^\top$ denotes the tangent part of $\x$, and $k\ge -m/2$ we obtain
	\begin{equation}\nonumber
	\vert\nabla\upeta\vert^2\le4\vert \x\vert^2=4(\upeta-2(m+2k)t)\le 4\upeta.
	\end{equation}
	We conclude
	\begin{eqnarray}
	\left(\dt-\Delta\right)\frac{\vert\f\vert^2}{\varmine}&\le&4\left(\varmine''\upeta-k\varmine'\right)\frac{\vert\f\vert^2}{\varmine^2}\le 0.\nonumber
	\end{eqnarray}
	The result then follows from Proposition \ref{prop 1}.
\end{proof}
\begin{example}
	Consider a polynomial
	$$\varmine(s):=\sum_{\ell=0}^ka_\ell s^\ell$$ 
	of degree $k\ge 0$, and assume that all coefficients are non-negative and $a_0>0$. 
	Then $\varmine >0$, $\varmine'\ge 0$ and $(s^{-k}\varmine')'\le 0$. This means that any polynomial growth rate of $\vert\f\vert^2$ will be preserved.
\end{example}
In \cite[Lemma 5.4]{EH} it was shown that a hypersurface in $\real{m+1}$ that is asymptotic to a cone will stay asymptotic to a cone under the mean curvature flow, provided each $\M_t$ has bounded curvature. We will now derive an analogue result in arbitrary codimension.
\begin{definition}
	Let $\M^m\subset\real{m+n}$ be a complete submanifold. We say that $\M^m$ is asymptotically conical at infinity, if there exist constants $\delta, c>0$ such that
	\begin{equation}
	\vert \F^\perp\vert^2\le c(1+\vert \F\vert^2)^{1-\delta}.\label{cond cone}
	\end{equation}
\end{definition}
Hence, for any sequence $\{x_k\}_{k\in\natural{}}\subset \M^m$ with $\lim_{k\to\infty}\vert\F(x_k)\vert=\infty$,  \eqref{cond cone} implies
$$\lim_{k\to\infty}\frac{\vert\F^\perp(x_k)\vert}{\vert\F(x_k)\vert}=0.$$
This means that the angle between $\F$ and its tangential part $\F^\top$ tends to zero.
\begin{lemma}\label{lemma 3.5}
	The following evolution equations for the normal part $\F^\perp$ of $\,\F$ hold.
	\begin{eqnarray}
	\Bigl(\nabla^\perp_{\dt}-\Delta^\perp\Bigr)\F^\perp&=&2\H+\langle A^{\F^\perp}\!,A\rangle,\nonumber\\
	\left(\dt-\Delta\right)\vert\F^\perp\vert^2&=&-2\vert\nabla^\perp\F^\perp\vert^2+4\langle \F^\perp,\H\rangle+2\vert A^{\F^\perp}\vert^2.\nonumber
	\end{eqnarray}
\end{lemma}
\begin{proof}We have,
	\begin{equation}\nonumber
	\nabla_{\dt}^\perp \F^\perp=\Bigl(\dt \F^\perp\Bigr)^\perp=\Bigl(\dt \Bigl(\F-\uu \gind ij\langle\F,\F_i\rangle \F_j\Bigr)\Bigr)^\perp=\H-\uu\gind ij\langle\F,\F_i\rangle\nabla_j^\perp \H=\H-\nabla_{\F^\top}^\perp\H.
	\end{equation}
	Similarly,
	\begin{equation}\nonumber
	\nabla_{k}^\perp \F^\perp=-\uu\gind ij\langle\F,\F_i\rangle A_{kj},
	\end{equation}
	and then with \textsc{Codazzi}'s equation
	\begin{eqnarray*}
		\Delta^\perp \F^\perp&=&-\uu\gind ij\bigl(\ud\delta ki+\langle\F,\ud Aki\rangle)A_{kj}-\uu\gind ij\langle\F,\F_i\rangle\nabla_j^\perp \H\\
		&=&-\H-\langle A^{\F^\perp}\!,A\rangle-\nabla^\perp_{\F^\top}\H.
	\end{eqnarray*}
	Combining this with the equation above proves the lemma.
\end{proof}

We want to show that \eqref{cond cone} is preserved under the mean curvature flow, if we allow the constant $c$ to vary in time. To achieve this we need the next lemma.
\begin{lemma}\label{lemma 3.6}
	Let $\uprho:=1+\vert\F\vert^2+2mt$. Then for any $\delta>0$ we have
	\begin{eqnarray}
	\left(\dt-\Delta\right)\uprho^{\delta-1}\vert\F^\perp\vert^2&\le&4(\vert A\vert^2+\delta^2)\uprho^{\delta-1}\vert\F^\perp\vert^2+2m\uprho^{\delta-1}.\nonumber
	\end{eqnarray}
\end{lemma}
\begin{proof}
	Let $\varmine>0$ be a function of $\uprho$ to be determined later. From 
	$$\left(\dt-\Delta\right)\uprho=0$$
	we derive
	\begin{eqnarray}
	\left(\dt-\Delta\right)\varmine\vert\F^\perp\vert^2&=&-\varmine''\vert\nabla\uprho\vert^2\vert\F^\perp\vert^2-2\varmine'\langle\nabla\uprho,\nabla\vert\F^\perp\vert^2\rangle\nonumber\\
	&&+2\varmine(\vert A^{\F^\perp}\vert^2+2\langle \F^\perp,\H\rangle -\vert\nabla^\perp\F^\perp\vert^2),\nonumber
	\end{eqnarray}
	and then
	\begin{eqnarray}
	-2\varmine'\langle\nabla\uprho,\nabla\vert\F^\perp\vert^2\rangle-2\varmine\vert\nabla^\perp\F^\perp\vert^2&\le&2\frac{(\varmine')^2}{\varmine}\vert\F^\perp\vert^2\vert\nabla\uprho\vert^2+2\varmine\bigl(\vert\nabla\vert\F^\perp\vert\vert^2-\vert\nabla^\perp\F^\perp\vert^2\bigr)\nonumber\\
	&\le&2\frac{(\varmine')^2}{\varmine}\vert\F^\perp\vert^2\vert\nabla\uprho\vert^2\nonumber
	\end{eqnarray}
	implies
	\begin{eqnarray}
	\left(\dt-\Delta\right)\varmine\vert\F^\perp\vert^2&\le&2\varmine(\vert A^{\F^\perp}\vert^2+2\langle \F^\perp,\H\rangle )+\left(2\Bigl(\frac{\varmine'}{\varmine}\Bigr)^2-\frac{\varmine''}{\varmine}\right)\vert\nabla\uprho\vert^2\varmine\vert\F^\perp\vert^2.\nonumber
	\end{eqnarray}
	For $\varmine(\uprho)=\uprho^{\delta-1}$, $\delta>0$, we get
	$$2\Bigl(\frac{\varmine'}{\varmine}\Bigr)^2-\frac{\varmine''}{\varmine}=\delta(\delta-1)\uprho^{-2}\le \delta^2\uprho^{-2}.$$
	Moreover, $\nabla\uprho=2\F^\top$ gives
	$$\vert\nabla\uprho\vert^2=4\vert\F^\top\vert^2\le 4\vert\F\vert^2\le 4\uprho^2,$$
	so that this together implies 
	\begin{eqnarray}
	\left(\dt-\Delta\right)\uprho^{\delta-1}\vert\F^\perp\vert^2&\le&2\uprho^{\delta-1}(\vert A^{\F^\perp}\vert^2+2\langle \F^\perp,\H\rangle )+4\delta^2\uprho^{\delta-1}\vert\F^\perp\vert^2.\nonumber
	\end{eqnarray}
	Finally, from the \textsc{Cauchy-Schwarz} inequality and from  $\vert\H\vert^2\le m\vert A\vert^2$ we conclude
	$$\vert A^{\F^\perp}\vert^2+2\langle \F^\perp,\H\rangle\le\vert A\vert^2\vert\F^\perp\vert^2+2\vert \H\vert \vert \F^\perp\vert\le 2\vert A\vert^2\vert\F^\perp\vert^2+m.$$
	This proves the proposition.
\end{proof}
\begin{proposition}
	Let $\M_t$ be a smooth solution of \eqref{mcf} with bounded curvature on each $\M_t$, and suppose $\M_0$ satisfies \eqref{cond cone} with $\delta\le 1$. Then for each $t$ there exists a constant $c(t)$ such that on $\M_t$ we have
	$$\vert \F^\perp\vert^2\le c(t)(1+\vert \F\vert^2)^{1-\delta}.$$
\end{proposition}
\begin{proof}
	Fix a time $t_0$. By assumption
	$$a_0:=\sup_{\M^m\times[0,t_0]}\vert A\vert<\infty.$$
	Lemma \ref{lemma 3.6} implies that on the time interval $[0,t_0]$ the function
	$$u:=\uprho^{\delta-1}\vert\F^\perp\vert^2$$
	satisfies the inequality
	$$\left(\dt-\Delta\right)u\le 4(a_0^2+\delta^2)u+2m\uprho^{\delta-1},$$
	where
	$$\uprho=1+\vert\F\vert^2+2mt.$$
	Since $\delta\le 1$ the last term on the right hand side of this inequality is bounded from above by $2m$ and we conclude
	$$\left(\dt-\Delta\right)u\le 4(a_0^2+\delta^2)u+2m.$$
	Thus $u$ can grow at most exponentially in time and there exists a constant $c(t_0)$ such 
	that
	$$\vert\F^\perp\vert^2\le c(t_0)(1+\vert\F\vert^2+2mt)^{1-\delta}\le c(t_0)(1+\vert\F\vert^2)^{1-\delta},$$
	for all $t\in[0,t_0]$.
\end{proof}

\section{Area decreasing maps in codimension two}\label{sec4}

A smooth map $\f:\real{m}\to \real{n}$ is called {\em area decreasing}, if 
$$|\D\f(v)\wedge \D\f(w)\vert<\vert v\wedge w\vert,\text { for all }v,w\in \real{m}.$$
In terms of the singular values of $\f$ this means that $\f$ is area decreasing, if and only if $\lambda_i\lambda_j<1$, for any $i\neq j$. The area decreasing property of $\f$ can be understood from a more geometric point of view. To this end consider the standard \textsc{Euclidean} and pseudo-\textsc{Riemannian} metrics $$\langle\cdot,\cdot\rangle_{\real{m+n}}=\dn\x^2+\dn\y^2,\quad \langle\cdot,\cdot\rangle_{\real{m,n}}=\dn\x^2-\dn\y^2$$
on $\real{m+n}$, where $\x,\y$ denote the cartesian coordinates on $\real{m}$, respectively on $\real{n}$. For any immersion $\F:\M^m\to\real{m+n}$ let us consider the symmetric bilinear forms $\gind,\sind$ on $\M^m$ defined by
$$\gind=\F^*\langle\cdot,\cdot\rangle_{\real{m+n}},\quad \sind=\F^*\langle\cdot,\cdot\rangle_{\real{m,n}}.$$
Hence $\gind$ is the induced metric, i.e. the first fundamental form on $\M$. The eigenvalues
$\sigma_1,\dots,\sigma_m$ of $\sind$ with respect to $\gind$ are bounded between $-1$ and 
$1$, and
$$\sigma_k>-1, k=1,\dots, m\quad\Leftrightarrow\quad\text{$\M$ is locally the graph over the $\x$-space $\real{m}$}.$$
$\sind$ is called \emph{two-positive}, if the sum of any two eigenvalues of $\sind$ is positive, i.e. if
$$\sigma_k+\sigma_l>0,\quad \text{for all}\quad k\neq l.$$
In particular, the two-positivity of $\sind$ implies that $\M$ is locally a graph. 
Note, that this does not follow, if we replace this by \emph{weakly two-positive}, i.e. by
$$\sigma_k+\sigma_l\ge 0,\quad \text{for all}\quad k\neq l.$$
If $\F(\x)=(\x,\f(\x))$ is a graph, then the eigenvalues of $\sind$ with respect to $\gind$ are related to the singular values of $\f$ by
$$\sigma_k=\frac{1-\lambda_k^2}{1+\lambda_k^2}, \quad k=1,\dots,m.$$
Since 
$$\sigma_k+\sigma_l=\frac{2\bigl(1-\lambda_k^2\lambda_l^2\bigr)}{\bigl(1+\lambda_k^2\bigr)\bigl(1+\lambda_l^2\bigr)}$$
we see that $\f$ is area decreasing, if and only if $\sind$ is two-positive. Hence
$$\sind\text{ is two-positive}\quad\Leftrightarrow\quad \M\text{ is locally the graph of an area decreasing map } \f.$$

\begin{remark}
	In contrast to the compact case, $\sup_{\M_t}\!\vert A\vert<\infty$ is not guaranteed for all $t\in[0,T)$. If $\f$ is uniformly area decreasing, then $\sup_{\M_0}\vert\D\f\vert^2<\infty$, and \eqref{mcf nonp} is uniformly parabolic. However,  if the map is merely area decreasing then $\vert\D\f\vert^2$ is not necessarily bounded, and the uniform parabolicity of \eqref{mcf nonp} might be violated.
\end{remark}
From now on let us assume that the codimension is $n=2$. On any submanifold
$\M\subset\real{m+2}$ we may consider the function 
$$\pind:=\frac{1}{2}\bigl(\operatorname{tr}_{\gind}\sind+2-m\bigr),$$
and since the eigenvalues of $\sind$ are bounded between $-1$ and $1$ we also get
\begin{equation}\nonumber
-1\le\pind\le 1.
\end{equation}
As was previously mentioned, $\pind>0$ on an $m$-dimensional submanifold $\M$ implies that $\M$ is locally the graph of an area decreasing map. However, $\pind\ge 0$ does not even imply that $\M$ is locally a graph. Instead, if $\pind\ge 0$ and $\M$ is locally a graph of a map $\f$, then $\f$ must be weakly area decreasing. In general, a map $\f:\real{m}\to\real{2}$ admits at most two non-trivial singular values
$\lambda$, $\mu$, and for the graph of $\f$ we obtain
$$\pind=\frac{1-\lambda^2\mu^2}{(1+\lambda^2)(1+\mu^2)}<1.$$
From
$$\f^*\langle\cdot,\cdot\rangle_{\real{n}}=\frac{1}{2}(\gind-\sind)$$
we deduce
$$\vert\nabla\f\vert^2=\operatorname{tr}_{\gind}\f^*\langle\cdot,\cdot\rangle_{\real{n}}=\frac{1}{2}(m-\operatorname{tr}_{\gind}\sind)=1-\pind,$$
where $\nabla\f$ is the gradient with respect to the induced metric $\gind$. Choosing $\varepsilon>0$, this yields the following characterization:
$$\begin{tabular}{r c l c l}
$\pind\ge 0$&$\Leftrightarrow$& $\vert\nabla\f\vert^2\le 1$&$\Leftrightarrow$&$\f $ is weakly area decreasing.\\
$\pind> 0$&$\Leftrightarrow$& $\vert\nabla\f\vert^2< 1$ &$\Leftrightarrow$&$\f $ is area decreasing.\\
$\pind\ge \varepsilon$&$\Leftrightarrow$& $\vert\nabla\f\vert^2\le 1-\epsilon$ &$\Leftrightarrow$&$\f $ is uniformly area decreasing.\\
\end{tabular}$$

The next theorem states that a weakly area decreasing map must either instantly become area decreasing or the flow splits off a linear subspace of dimension $m-2$.
\begin{theorem}\label{weakly}
	If $\,\M_t$ is a smooth solution of \eqref{mcf} with bounded gradient and bounded curvature on each $\M_t$, and if $\,\M_0$ is the graph of a weakly area decreasing map, then one and only one of the following two cases holds:
	\vspace*{4pt}
	\begin{enumerate}[\normalfont (1)]
		\item $\M_t$ instantly becomes the graph of an area decreasing map for $t>0$.
		\vspace{4pt}
		\item $\M_t$ splits into $\Sigma_t\times\real{m-2}$, where $\Sigma_t\subset \complex{2}$ is the graph of a symplectomorphism $\phi_t:\complex{}\to\complex{}$ that evolves by the Lagrangian mean curvature flow.
	\end{enumerate}
\end{theorem}
Before we can prove this theorem we need to recall some of the notations
and results in \cite{ASS22}.
We start with the singular value decomposition
and follow the exposition in \cite{SS14}, \cite{SS15} and \cite{SS18}.
Let $\f:\real{m}\to\real{2}$ be smooth, fix a point $\x\in \real{m}$, and let
$$\lambda_{1}\ge\displaystyle{\dots}\ge\lambda_{m}$$
denote the singular values of
$\f^{*}\langle\cdot,\cdot\rangle_{\real{2}}$ at $\x$ with respect to $\langle\cdot,\cdot\rangle_{\real{m}}$. Since the rank of $\D\f$ is at most $2$ the singular values of $\D\f$ at $\x$
are given by
$$\lambda:=\lambda_1,\quad\mu:=\lambda_2,\quad\text{ and }\lambda_k=0,\,\text{ for }k\ge 3.$$
There exists an orthonormal basis $\{\alpha_{1},\dots,\alpha_{m}\}$ of $\real{m}$ and an orthonormal basis $\beta_1,\beta_2$ of $\real{2}$ such that at $\x$
$$\D\f(\alpha_1)=\lambda\beta_1,\quad \D\f(\alpha_2)=\mu\beta_2,\quad\D\f(\alpha_k)=0,\,\text{ for } k\ge 3.$$
Therefore, at $\x$ we have with respect to this basis 
$$\gind=\operatorname{diag}\bigl(1+\lambda^2,1+\mu^2,0,\dots,0\bigr),\quad \sind=\operatorname{diag}\bigl(1-\lambda^2,1-\mu^2,0,\dots,0\bigr).$$
Then the frame $\{e_1,\dots,e_m\}$ given by
\begin{equation*}
e_1:=\frac{\alpha_1}{\sqrt{1+\lambda^2}}, \quad e_2:=\frac{\alpha_2}{\sqrt{1+\mu^2}},\quad e_{i}:=\alpha_i,\,\,\text{for }i\ge 3,\label{tangent}
\end{equation*}
forms an orthonormal basis of $\tmax_{\x}\M$ with respect to the metric $\gind$ and $\{\xi,\upeta\}$ given by
\begin{equation*}
\xi:=\frac{-\lambda\alpha_1\oplus\beta_1}{\sqrt{1+\lambda^2}},\quad
\upeta:=\frac{-\mu\alpha_2\oplus\beta_2}{\sqrt{1+\mu^2}},\label{normal}
\end{equation*}
forms an orthonormal basis of the normal space $\tmax_{\x}^\perp \M$ of $\M$ at $\F(\x)=(\x,\f(\x))$.
With respect to the orthonormal frame $\{e_1,\dots,e_m\}$ the bilinear form $\sind$ satisfies
\begin{equation}\nonumber
\left(\sind(e_{i},e_{j})\right)_{i,j}=\operatorname{diag}\left(\frac{1-\lambda^{2}}{1+\lambda^{2}},\frac{1-\mu^{2}}{1+\mu^{2}},1,\dots,1\right).
\end{equation}
The restriction $\sind^\perp$ of the pseudo-\textsc{Riemannian} metric $\langle\cdot,\cdot\rangle_{\real{m,2}}$ onto the normal bundle of the graph satisfies the identities
\begin{eqnarray}
\sind^{\perp}(\xi,\xi)=\frac{\lambda^{2}-1}{1+\lambda^{2}},\quad \sind^{\perp}(\upeta,\upeta)=\frac{\mu^{2}-1}{1+\mu^{2}},\quad \sind^{\perp}(\xi,\upeta)=0.\nonumber
\end{eqnarray}
Then
\begin{align}
&\langle\D\F(e_{k}),\xi\rangle_{\real{m,2}}=\langle\D\F(e_{k}),\upeta\rangle_{\real{m,2}}=0,\,\text{  for }k\ge 3,\label{mixeda}\\ &\langle\D\F(e_{1}),\upeta\rangle_{\real{m,2}}=\langle\D\F(e_{2}),\xi\rangle_{\real{m,2}}=0.\label{mixedb}
\end{align}
Define now the quantities $\tmax_{11}=\langle\D\F(e_{1}),\xi\rangle_{\real{m,2}}$ and $\tmax_{22}=\langle\D\F(e_{2}),\upeta\rangle_{\real{m,2}}$. Then
\begin{equation}\nonumber
\sind_{11}=\frac{1-\lambda^2}{1+\lambda^2},\quad\tmax_{11}=-\frac{2\lambda}{1+\lambda^{2}},\quad\text{and}\quad\sind_{22}=\frac{1-\mu^2}{1+\mu^2},\quad
\tmax_{22}=-\frac{2\mu}{1+\mu^{2}},
\end{equation}
so that
$$\sind_{kk}^2+\tmax_{kk}^2=1,\quad k=1,2.$$
With respect to the orthonormal bases $\{e_1,\dots,e_m\}$, $\{\xi,\upeta\}$ the components of the second fundamental form $A$ will be denoted by
$$A^\xi_{kl}:=\langle A(e_k,e_l),\xi\rangle_{\real{m+2}},\quad A^\upeta_{kl}:=\langle A(e_k,e_l),\upeta\rangle_{\real{m+2}}.$$
\begin{lemma}\label{evolpind}
	With respect to the singular value decomposition $\pind$ satisfies 
	\begin{eqnarray}
	\nabla_{e_k}\pind&=&A^\xi_{1k}\tind_{11}+A^\upeta_{2k}\tind_{22},\nonumber\\
	\left(\dt-\Delta\right)\pind&=&2\pind\vert A\vert^2+\sum_{k=1, i=3}^m\vert A^\xi_{ki}\vert ^2(1-\sind_{11})+\sum_{k=1, i=3}^m\vert A^\upeta_{ki}\vert ^2(1-\sind_{22})\nonumber\\
	&&+\underbrace{\sum_{k=1}^m\big(\vert A^\xi_{1k}\vert^2-\vert A^\upeta_{2k}\vert^2\big)(\sind_{11}-\sind_{22})}_{=:\mathcal{B}},\nonumber
	\end{eqnarray}
	where $\mathcal{B}$ satisfies
	\begin{equation}
	2\pind \mathcal{B}=-\vert\nabla\pind\vert^2+\sum_{k=1}^m\vert A^\xi_{1k}\tind_{22}+A^\upeta_{2k}\tind_{11}\vert^2\label{p4}.
	\end{equation}
	Moreover, in case $\pind>0$ the function $\qind:=\sqrt{\pind}$ is well defined and smooth, and satisfies
	\begin{equation}
	\left(\dt-\Delta\right)\qind\ge\qind\vert A\vert^2.\label{p10}
	\end{equation}
\end{lemma}
\begin{proof}
	Up to a factor $2$ the function $\pind$ coincides with the function given in \cite[Lemma 4.1]{ASS22}. Therefore the statement follows from the computations in that
	lemma, for more details see \cite[eqns (4.6)-(4.8)]{ASS22}.
\end{proof}
Since we work in a non-compact setting we cannot use \eqref{p10} directly, because the function $\pind$ might instantly vanish or become even negative at some points for $t>0$. Therefore, we need a better variant of the evolution equation for $\pind$. The idea is to replace the term $\mathcal{B}$ in Lemma \ref{evolpind} by a gradient term and another quantity that depends on $\pind$. It is straightforward to verify
\begin{eqnarray}
2-\sind_{11}^2-\sind_{22}^2&=&2(1+\sind_{11})(1+\sind_{22})-4\pind(\pind+1).\label{ind}\\
(1+\sind_{11})(1+\sind_{22})&=&\frac{4}{(1+\lambda^2)(1+\mu^2)}\ge\frac{4}{(1+\vert\D\f\vert^2)^2}.\label{ind2}
\end{eqnarray}
For the vector field
$$X=\sum_{k=1}^mX_ke_k,\quad\text{where}\quad X_k:=A^\xi_{1k}\tmax_{11}-A^\upeta_{2k}\tmax_{22},$$
we compute with Lemma \ref{evolpind}
\begin{eqnarray}
2\langle\nabla\pind,X\rangle&=&2\sum_{k=1}^m\vert A^\xi_{1k}\vert^2\tmax_{11}^2-2\sum_{k=1}^m\vert A^\upeta_{2k}\vert^2\tmax_{22}^2\nonumber\\
&=&(\tmax_{11}^2+\tmax_{22}^2)\sum_{k=1}^m\bigl(\vert A^\xi_{1k}\vert^2-\vert A^\upeta_{2k}\vert^2\bigr)+(\tmax_{11}^2-\tmax_{22}^2)\sum_{k=1}^m\bigl(\vert A^\xi_{1k}\vert^2+\vert A^\upeta_{2k}\vert^2\bigr),\nonumber
\end{eqnarray}
and with
$$\sind_{kk}^2+\tind_{kk}^2=1$$
this leads to
\begin{eqnarray}
2\langle\nabla\pind,X\rangle&=&(2-\sind_{11}^2-\sind_{22}^2)\sum_{k=1}^m\bigl(\vert A^\xi_{1k}\vert^2-\vert A^\upeta_{2k}\vert^2\bigr)+(\sind^2_{22}-\sind^2_{11})\sum_{k=1}^m\bigl(\vert A^\xi_{1k}\vert^2+\vert A^\upeta_{2k}\vert^2\bigr)\nonumber\\
&\overset{(\ref{ind})}{=}&2(1+\sind_{11})(1+\sind_{22})\sum_{k=1}^m\bigl(\vert A^\xi_{1k}\vert^2-\vert A^\upeta_{2k}\vert^2\bigr)\nonumber\\
&&-4\pind(\pind+1)\sum_{k=1}^m\bigl(\vert A^\xi_{1k}\vert^2-\vert A^\upeta_{2k}\vert^2\bigr)+2\pind(\sind_{22}-\sind_{11})\sum_{k=1}^m\bigl(\vert A^\xi_{1k}\vert^2+\vert A^\upeta_{2k}\vert^2\bigr)\nonumber\\
&=&2(1+\sind_{11})(1+\sind_{22})\sum_{k=1}^m\bigl(\vert A^\xi_{1k}\vert^2-\vert A^\upeta_{2k}\vert^2\bigr)\nonumber\\
&&+4\pind\left((1+\sind_{22})\sum_{k=1}^m\vert A^\upeta_{2k}\vert^2-(1+\sind_{11})\sum_{k=1}^m\vert A^\xi_{1k}\vert^2\right).\nonumber
\end{eqnarray}
We may use
$$1+\sind_{11}>0\quad \text{and} \quad 1+\sind_{22}>0,$$ because all singular values of $\f$ are finite as long as we have a graphical solution. This implies the following identity for graphs in codimension two:
\begin{eqnarray*}
	\mathcal{B}&=&\frac{\sind_{11}-\sind_{22}}{(1+\sind_{11})(1+\sind_{22})}\langle\nabla\pind,X\rangle\\
	&&+2\pind\left(\frac{\sind_{22}-\sind_{11}}{1+\sind_{11}}\sum_{k=1}^m\vert A^\upeta_{2k}\vert^2+\frac{\sind_{11}-\sind_{22}}{1+\sind_{22}}\sum_{k=1}^m\vert A^\xi_{1k}\vert^2\right).
\end{eqnarray*}
Note also that from
\begin{equation}\label{grad s}
\nabla_{e_k}\sind(e_i,e_j)=\langle A(e_k,e_i),\D\F(e_j)\rangle_{\real{m,2}}+\langle A(e_k,e_j),\D\F(e_i)\rangle_{\real{m,2}}
\end{equation}
we get
\begin{eqnarray*}
	\nabla_{e_k}(\sind_{11}-\sind_{22})^2&=&4(\sind_{11}-\sind_{22})\bigl(\langle A(e_k,e_1),\D\F(e_1)\rangle_{\real{m,2}}-\langle A(e_k,e_2),\D\F(e_
	2)\rangle_{\real{m,2}}\bigr)\\
	&=&4(\sind_{11}-\sind_{22})(A^\xi_{1k}\tind_{11}-A^\upeta_{2k}\tind_{22})\\
	&=&4(\sind_{11}-\sind_{22})X_k.
\end{eqnarray*}
Summarizing, we have shown the following:
\begin{lemma}\label{lemma 4.4}
	With respect to the singular value decomposition the function $\pind$ satisfies
	\begin{eqnarray}
	\left(\dt-\Delta\right)\pind&=&\langle\a,\nabla\pind\rangle +2\pind\left(\vert A\vert^2+\frac{\sind_{22}-\sind_{11}}{1+\sind_{11}}\sum_{k=1}^m\vert A^\upeta_{2k}\vert^2+\frac{\sind_{11}-\sind_{22}}{1+\sind_{22}}\sum_{k=1}^m\vert A^\xi_{1k}\vert^2\right)\nonumber\\
	&&+\sum_{k=1, i=3}^m\vert A^\xi_{ki}\vert ^2(1-\sind_{11})+\sum_{k=1, i=3}^m\vert A^\upeta_{ki}\vert ^2(1-\sind_{22}),\nonumber
	\end{eqnarray}
	where $\a$ is the smooth vector field given by
	\begin{equation}
	\a=\frac{\nabla(\sind_{11}-\sind_{22})^2}{4(1+\sind_{11})(1+\sind_{22})}.\label{vec a}
	\end{equation}
\end{lemma}
We are now ready to prove that the weakly area decreasing property is preserved.

{\em Proof of Theorem $4.2$.}
	Suppose that the solution of \eqref{mcf} is given on $[0,t_0)$. For graphs we always have
	$$1\ge\sind_{kk}>-1, \quad k=1,2,$$
	and we can use the evolution equation of $\pind$ in Lemma \ref{lemma 4.4} to estimate
	\begin{eqnarray}
	\left(\dt-\Delta\right)\pind&\ge&\langle\a,\nabla\pind\rangle +2\pind\left(\vert A\vert^2+\frac{\sind_{22}-\sind_{11}}{1+\sind_{11}}\sum_{k=1}^m\vert A^\upeta_{2k}\vert^2+\frac{\sind_{11}-\sind_{22}}{1+\sind_{22}}\sum_{k=1}^m\vert A^\xi_{1k}\vert^2\right).\nonumber
	\end{eqnarray}
	From \eqref{ind2}, \eqref{grad s}, and \eqref{vec a} we obtain that the vector field $\a$ is uniformly bounded on each finite time interval $[0,t_1]$, $t_1\in(0,t_0)$. On the other hand
	\begin{eqnarray}
	\vert A\vert^2&+&\frac{\sind_{22}-\sind_{11}}{1+\sind_{11}}\sum_{k=1}^m\vert A^\upeta_{2k}\vert^2+\frac{\sind_{11}-\sind_{22}}{1+\sind_{22}}\sum_{k=1}^m\vert A^\xi_{1k}\vert^2\nonumber\\
	&\ge&\sum_{k=1}^m\vert A^\xi_{1k}\vert^2+\sum_{k=1}^m\vert A^\upeta_{2k}\vert^2+\frac{\sind_{22}-\sind_{11}}{1+\sind_{11}}\sum_{k=1}^m\vert A^\upeta_{2k}\vert^2+\frac{\sind_{11}-\sind_{22}}{1+\sind_{22}}\sum_{k=1}^m\vert A^\xi_{1k}\vert^2\nonumber\\
	&=&\frac{1+\sind_{22}}{1+\sind_{11}}\sum_{k=1}^m\vert A^\upeta_{2k}\vert^2+\frac{1+\sind_{11}}{1+\sind_{22}}\sum_{k=1}^m\vert A^\xi_{1k}\vert^2\ge 0,\nonumber
	\end{eqnarray}
	so that 
	\begin{equation}\label{strong}
	\left(\dt-\Delta\right)\pind\ge\langle\a,\nabla\pind\rangle +\b \pind.
	\end{equation}
	for the function
	$$\b:=2\left(\vert A\vert^2+\frac{\sind_{22}-\sind_{11}}{1+\sind_{11}}\sum_{k=1}^m\vert A^\upeta_{2k}\vert^2+\frac{\sind_{11}-\sind_{22}}{1+\sind_{22}}\sum_{k=1}^m\vert A^\xi_{1k}\vert^2\right)\ge 0.$$
	We may therefore apply the maximum principle derived in Proposition \ref{prop 2} to conclude $\pind\ge 0$, for all $t\in[0,t_1]$. This shows $\f$ stays weakly area decreasing. Furthermore, on $[0,t_1]\subset [0,t_0)$ the singular values are uniformly bounded, the \textsc{Laplace-Beltrami} operator $\Delta$ stays uniformly elliptic, and $\a$ stays uniformly bounded. Thus we can apply the standard strong parabolic maximum principle to \eqref{strong} to conclude that either $\pind>0$ for $t\in(0,t_0)$ or $\pind\equiv 0$ for all $t\ge 0$. In the second case we derive from $\lambda^2\mu^2=1$ that $1-\sind_{kk}>0$, $k=1,2$. Then Lemma \ref{lemma 4.4} implies $A_{ki}=0$, for $k\ge 1, i\ge 3$. We continue in two steps.
	
	\vspace*{10pt}
	(i) Fix a time $t$, and let $\sind^\sharp$ denote the self-adjoint endomorphism induced by $\sind$, i.e.
	$$\gind(\sind^\sharp v,w)=\sind(v,w),\, \text{ for all }v,w\in \tmax\M.$$ 
	We claim that the distributions $\mathcal{D}_t:=\operatorname{ker}(\sind^\sharp-\operatorname{Id})$ and its orthogonal complement $\mathcal{D}_t^\perp$ are invariant under parallel transport. Using the singular value decomposition, we see that the fibers of $\mathcal{D}_t$ are spanned by the vectors $e_3,\dots,e_m$ and then $A_{ki}=0$, for all $k\ge 1, i\ge 3$ means
	\begin{equation}\label{invol 1}
	A(v,w)=0,\,\text{ for all } v\in\tmax\M, w\in \mathcal{D}_t.
	\end{equation} 
	Since
	$$\gind\bigl((\nabla_v\sind^\sharp)w,u\bigr)=\nabla_{v}\sind(w,u)=\langle A(v,w),\D\F(u)\rangle_{\real{m,2}}+\langle A(v,u),\D\F(w)\rangle_{\real{m,2}},$$
	equations \eqref{mixeda}, \eqref{mixedb} and \eqref{invol 1} imply
	\begin{equation}\label{invol 2a}
	\gind\bigl((\nabla_v\sind^\sharp)w,u\bigr)=0,\,\text{ if one of the vectors }v, w, u \text{ lies in }\mathcal{D}_t.
	\end{equation}
	Now let $V\in\Gamma(\tmax\M)$, $W\in\Gamma(\mathcal{D}_t)$, $U\in\Gamma(\mathcal{D}_t^\perp)$ be smooth sections. 
	\begin{eqnarray}
	\gind(\nabla_VW,U)&\overset{W\in\Gamma(\mathcal{D}_t)}{=}&\gind(\nabla_V(\sind^\sharp W),U)=\gind\bigl((\nabla_V\sind^\sharp)W+\sind^\sharp(\nabla_VW),U\bigr)\nonumber\\
	&\overset{(\ref{invol 2a})}{=}&\gind(\sind^\sharp(\nabla_VW),U).\nonumber
	\end{eqnarray}
	Since this holds for any $U\in\Gamma(\mathcal{D}_t^\perp)$ we must have
	$$\nabla_VW-\sind^\sharp(\nabla_VW)\in\Gamma(\mathcal{D}_t).$$
	Decomposing 
	$$\nabla_VW=X_1+X_2,$$
	where $X_1\in\Gamma(\mathcal{D}_t^\perp), X_2\in\Gamma(\mathcal{D}_t)$, we obtain from $\sind^\sharp X_2=X_2$ that $X_1-\sind^\sharp X_1\in\Gamma(\mathcal{D}_t)$. On the other hand this lies in $\Gamma(\mathcal{D}_t^\perp)$, so that $X_1\equiv 0$, i.e. we have shown
	\begin{equation}\nonumber
	\nabla_VW\in\Gamma(\mathcal{D}_t),\,\text{ for all }V\in\Gamma(\tmax\M), W\in\Gamma(\mathcal{D}_t).
	\end{equation}
	But then also
	$$0=V\gind(W,U)=\gind(\nabla_VW,U)+\gind(W,\nabla_VU)=\gind(W,\nabla_VU).$$
	Again, this holds for any $V\in\Gamma(\tmax\M)$, $W\in\Gamma(\mathcal{D}_t)$, $U\in\Gamma(\mathcal{D}_t^\perp)$, from which we conclude that
	\begin{equation}\nonumber
	\nabla_VU\in\Gamma(\mathcal{D}_t^\perp),\,\text{ for all }V\in\Gamma(\tmax\M), U\in\Gamma(\mathcal{D}_t^\perp).
	\end{equation}
Hence $\mathcal{D}_t, \mathcal{D}_t^\perp$ are both invariant under parallel transport. It follows 
from the \textsc{de Rham} decomposition theorem and the fact that $A$ vanishes on $\mathcal{D}_t$ 
that $\M_{t}$ locally splits into the \textsc{Riemannian} product of a surface $\Sigma_{t}$ and a totally 
geodesic subset $E_{t}$ of dimension $m-2$. The submanifold $\M_{t}$ is simply connected since it is 
an entire graph over $\real{m}$. Therefore the splitting is global. Since
$\mathcal{D}_t=\operatorname{ker}\D\f\vert_t$ and $E_t$ is totally geodesic we conclude the existence 
of an isomorphism $L_t:\complex{}\times\real{m-2}\to\real{2m}$ and of a smooth map
$\phi_t:\complex{} \to \complex{}$ such that $\tmax\complex{}=L_t^*\mathcal{D}_t^\perp$,
$\tmax\real{m-2}=L_t^*\mathcal{D}_t$, and
$$\f(L_t(z,\zeta),t)=\phi_t(z),$$ for all $(z,\zeta)\in \complex{}\times \real{m-2}.$
Since $\pind=0$ the map $\phi_t:\complex{} \to \complex{}$ is area preserving which is equivalent to $\phi_t$ being a symplectomorphism with respect to the standard symplectic structure $\omega$ on $\complex{}$. Hence the graph of $\phi_t$ is \textsc{Lagrangian} in $\bigl(\complex{2},\omega\ominus \omega\bigr)$,
	
	\vspace*{10pt}
	(ii) The splitting in (i) shows that at each time $t$ we have 
	$$\tmax\M=\mathcal{D}_t\oplus\mathcal{D}^\perp_t,$$
	but in principle the splitting might depend on $t$. We will now show that this is not the case, and that $\mathcal{D}_t, \mathcal{D}_t^\perp$ are indeed time-independent. To proceed, let us consider the space-time manifold $\overline{\M}:=\M^m\times[0,\tmax)$. Moreover, let $\mathcal{D}$,
	$\mathcal{D}^\perp$ denote the subbundles of $\tmax\overline{\M}$ with fibers
	$$\mathcal{D}_{(x,t)}=(\mathcal{D}_t)_x\oplus 0,\quad \mathcal{D}^\perp_{(x,t)}=(\mathcal{D}^\perp_t)_x\oplus 0,\quad\text{ for all }(x,t)\in \M^m\times[0,\tmax).$$
	To prove that $\mathcal{D}_t$ and $\mathcal{D}^\perp_t$ are time-independent we need to show that they are invariant with respect to covariant derivative in time, i.e. we need to show that
	$$\nabla_{\dt} W\in\Gamma(\mathcal{D}),\,\text{ for all }W\in\Gamma(\mathcal{D}),\quad \text{and} \quad \nabla_{\dt} U\in\Gamma(\mathcal{D}^\perp),\,\text{ for all }U\in\Gamma(\mathcal{D}^\perp),$$
	so that the result follows from \textsc{de Rhams} decomposition theorem.
	Choose arbitrary $V_1,V_2\in\Gamma(\tmax\M)$ and $W\in\Gamma(\mathcal{D})$. Then $\nabla_{V_1}W\in\Gamma(\mathcal{D})$, so that
	$$A(W,\cdot)=A(\nabla_{V_1}W,\cdot)=0$$
	and
	\begin{eqnarray}
	(\nabla_{V_1}A)(W,V_2)=V_1\bigl(A(W,V_2)\bigr)-A(\nabla_{V_1}W,V_2)-A(W,\nabla_{V_1}V_2)=0.\label{invol 5}
	\end{eqnarray}
	From \textsc{Codazzi}'s equation we obtain for all $w\in\mathcal{D}$, we have
	\begin{equation}\label{invol 6}
	\nabla_w\H=\sum_{k=1}^m(\nabla_wA)(e_k,e_k)\overset{\textsc{Codazzi}}{=}\sum_{k=1}^m(\nabla_{e_k}A)(w,e_k)\overset{(\ref{invol 5})}{=}0.
	\end{equation}
	Since
	$$\sind(V,W)=\langle\D\F(V),\D\F(W)\rangle_{\real{m,2}}\quad\text{and}\quad \dt \F=\H$$
	we compute for $V\in\Gamma(\tmax\M)$, $W\in\Gamma(\mathcal{D})$, that
	$$\nabla_{\dt}\!\!\sind(V,W)=\langle\nabla_V\H,\D\F(W)\rangle_{\real{m,2}}+\langle\nabla_W\H,\D\F(V)\rangle_{\real{m,2}}\overset{(\ref{invol 6})}{=}\langle\nabla_V\H,\D\F(W)\rangle_{\real{m,2}}.$$
	Now
	$$\langle\nabla_V\H,\D\F(W)\rangle_{\real{m,2}}=V\langle \H,\D\F(W)\rangle_{\real{m,2}}-\langle \H,A(V,W)\rangle_{\real{m,2}}.$$
	Equation \eqref{mixeda} and $W\in\Gamma(\mathcal{D})$ imply that
	$$\langle \H,\D\F(W)\rangle_{\real{m,2}}=0.$$
	From \eqref{invol 1} we see that
	$$\langle \H,A(V,W)\rangle_{\real{m,2}}=0,$$
	so that
	\begin{equation}\nonumber
	\nabla_{\dt}\!\!\sind(v,w)=0,\quad\text{ for all }v\in \tmax\M,\,w\in\mathcal{D}.
	\end{equation}
	Proceeding as in part (i) it follows that
	$$\gind \Bigl(W,\nabla_{\dt}U-\sind^\sharp(\nabla_{\dt}U)\Bigr)=0,\text{ for all }W\in\Gamma(\mathcal{D}), U\in\Gamma(\mathcal{D}^\perp),$$
	which again implies
	$$\nabla_{\dt}U\in\Gamma(\mathcal{D}^\perp).$$
	Likewise we derive
	$$\nabla_{\dt}W\in\Gamma(\mathcal{D}),\quad\text{for all}\quad W\in\Gamma(\mathcal{D}).$$ 
Therefore the splitting of $\tmax\M$ is global and time-independent, and the aforementioned 
isomorphism $L_t:\complex{}\times\real{m-2}\to\real{2m}$ is time-independent as well. This implies now that the symplectomorphisms $\phi_t$ 
evolve by the Lagrangian mean curvature flow, and the proof is completed.
$\hfill \square$

If $\f$ is area decreasing but not uniformly area decreasing, $\pind$ might satisfy some lower  bounds in terms of $\vert\x\vert$. We will now prove that certain bounds for $\pind$ are preserved.
\begin{proposition}[Decay estimates for $\pind$]\label{lemm 4}
	Let $\M_t$ be a smooth solution of \eqref{mcf} with bounded gradient and bounded curvature on each $\M_t$. Suppose $\M_0$ is the graph of an area decreasing map with
	$$\pind\ge  e^{\displaystyle-\varmine(\vert\x\vert^2)},$$
	where $\varmine:[0,\infty)\to\real{}$ is a smooth function such that for some constant $c\ge 0$ and all $s>0$ we have
	$$0\le\varmine'(s)\le \frac{c}{\sqrt{s}},\quad{ and }\quad \varmine''(s)\le 0.$$
	Then $\M_t$ is the graph of an area decreasing map with
	\begin{equation}
	\pind\ge e^{\displaystyle-\varmine(\vert\x\vert^2+2mt)}.\nonumber
	\end{equation}
\end{proposition}
\begin{proof}
	In the definition of $\upeta$ in \eqref{defeta} we set $\uplambda(t)=0$ so that
	$$\upeta=\vert\x\vert^2+2mt\ge 0.$$
	In the sequel we will simply write $\varmine$ instead of $\varmine\circ\upeta$. From Theorem \ref{weakly} follows that $\M_t$ is the graph of an area decreasing map. Thus $\pind>0$, and the function $\qind=\sqrt{\pind}$ is well defined and smooth. From
	\begin{eqnarray}
	\left(\dt-\Delta\right) \qind e^{\varmine/2}&=&e^{\varmine/2}\left(\dt-\Delta\right)\qind+\qind\left(\dt-\Delta\right)e^{\varmine/2}-2\langle\nabla\qind,\nabla e^{\varmine/2}\rangle,\nonumber
	\end{eqnarray}
	and from Lemma \ref{lemma 3.1}, and \eqref{p10} follows the estimate
	\begin{eqnarray}
	&\ge&\qind e^{\varmine/2}\vert A\vert^2+\frac{1}{2}\qind e^{\varmine/2}\left(\dt-\Delta\right)\varmine-\frac{1}{4}\qind e^{\varmine/2}\vert\nabla\varmine\vert^2-e^{\varmine/2}\langle\nabla\qind,\nabla \varmine\rangle\nonumber\\
	&{\ge}&\qind e^{\varmine/2}\vert A\vert^2-\langle\nabla(\qind e^{\varmine/2}),\nabla\varmine\rangle+\frac{1}{4}\qind e^{\varmine/2}\vert\nabla\varmine\vert^2\ge-\langle\nabla(\qind e^{\varmine/2}),\nabla\varmine\rangle.\nonumber
	\end{eqnarray}
	The result then is a consequence of the maximum principle in Proposition \ref{prop 1} since
	$$\vert\nabla\varmine\vert^2=(\varmine'(\upeta))^2\vert\nabla\upeta\vert^2\le4\upeta(\varmine'(\upeta))^2,
	$$
and
$$s(\varmine'(s))^2\le c.$$
This completes the proof.
\end{proof}

As we will now see, this proposition implies that polynomial and exponential decay rates for $\pind$ are preserved and that uniformly area decreasing maps stay uniformly area decreasing.
\begin{corollary}\label{area uni}
	Under the assumptions made in Lemma \ref{lemm 4} assume that either one of the following estimates holds on $\M_0\,$:
	$$\text{\rm(a)}\quad\pind\ge \varepsilon e^{-\sigma\sqrt{\vert\x\vert^2+1}},\quad\text{\rm(b)}\quad\pind\ge \frac{\varepsilon}{(\vert\x\vert^2+1)^k},\quad\text{\rm(c)}\quad\pind\ge \frac{\varepsilon}{\ln(\vert\x\vert^2+1+\sigma)},$$
	where $\varepsilon>0$, $\sigma> 0$, $k\ge 0$ are constants. Then for $t>0$ we have
	\begin{gather*}
	\text{\rm(a)}\quad\pind\ge \varepsilon e^{-\sigma\sqrt{\vert\x\vert^2+1+2mt}},\quad\text{\rm(b)}\quad\pind\ge \frac{\varepsilon}{(\vert\x\vert^2+1+2mt)^k},\\
	\text{\rm(c)}\quad\pind\ge \frac{\varepsilon}{\ln(\vert\x\vert^2+1+\sigma+2mt)}.
	\end{gather*}
	In particular, the uniformly area decreasing property $\pind\ge\varepsilon>0$ is preserved.
\end{corollary}
\begin{proof}
	The functions $\varmine_{\rm a}(s):=\sigma{\sqrt{s+1}}-\ln \varepsilon$, $\varmine_{\rm b}(s):=k\ln{\sqrt{s+1}}-\ln \varepsilon$, and $\varmine_{\rm c}(s)=\ln(\ln(s+1+\sigma))-\ln \varepsilon$
	satisfy all assumptions in Lemma \ref{lemm 4}. Choosing $k=0$ in example (b) implies that uniformly area decreasing maps stay uniformly area decreasing with the same lower bound for $\pind$.
\end{proof}

\section{Longtime existence}\label{sec5}
From now on we shall only consider the case where $\M_0$ is the graph of a uniformly area decreasing map in codimension two, i.e. we assume that for some fixed constant $\varepsilon>0$ the inequality
\begin{equation}\label{uniform}
\pind\ge\varepsilon
\end{equation}
holds on all of $\M_0$. In addition we shall always assume that $\vert\nabla^\ell A\vert$ is uniformly bounded on $\M_0$ for all $\ell\ge 0$. Since $\pind\ge\varepsilon$ implies that $\vert\D\f\vert<\infty$ the short time existence result in Proposition \ref{ste} and Corollary \ref{area uni} then ensure that \eqref{uniform} remains valid for $t> 0$, as long as the second fundamental form remains bounded on each $\M_t$. The purpose of this section is to obtain a priori estimates for the curvature and its covariant derivatives depending only on the initial conditions, and that will guarantee longtime existence. We start with an estimate for the mean curvature.

\noindent
\begin{lemma}\label{lemma 5.1}
	For $\pind>0$ the function ${\vert\H\vert^2}{\pind}^{-1}$ satisfies
	\begin{equation*}
	\left(\dt-\Delta\right)\frac{\vert\H\vert^2}{\pind}\le\frac{1}{\pind}\Bigl\langle\nabla\pind,\nabla\frac{\vert\H\vert^2}{\pind}\Bigr\rangle.
	\end{equation*}
\end{lemma}
\begin{proof}
From the \textsc{Cauchy-Schwarz} inequality we get
\begin{equation}\label{mean est}
\vert\nabla\vert\H{\vert^2\vert^2\le }4\vert\H\vert^2\vert\nabla^\perp\H\vert^2,\quad \vert A^{\H}\vert^2\le\vert\H\vert^2\vert A\vert^2.
\end{equation}
If $\pind >0$, Lemma \ref{evolpind} and \eqref{p4} imply 
\begin{equation}\label{p11}
\left(\dt-\Delta\right)\pind\ge 2\pind\vert A\vert^2-\frac{1}{2\pind}\vert\nabla\pind\vert^2.
\end{equation}
Then Lemma \ref{lemma 2.2} and 
\begin{eqnarray}
\left(\dt-\Delta\right)\frac{\vert\H\vert^2}{\pind}&=&\frac{1}{\pind}\left(\dt-\Delta\right)\vert\H\vert^2-\frac{\vert\H\vert^2}{\pind^2}\left(\dt-\Delta\right)\pind +\frac{2}{\pind}\left\langle\nabla\pind,\nabla\frac{\vert\H\vert^2}{\pind}\right\rangle\nonumber
\end{eqnarray}
lead to
\begin{eqnarray}
\left(\dt-\Delta\right)\frac{\vert\H\vert^2}{\pind}&\le&-\frac{2}{\pind}\vert\nabla^\perp\H\vert^2+2\frac{\vert\H\vert^2}{\pind}\vert A\vert^2-\frac{\vert\H\vert^2}{\pind^2}\left(2\pind\vert A\vert^2-\frac{1}{2\pind}\vert\nabla\pind\vert^2\right)\nonumber\\
&&+\frac{2}{\pind}\left\langle\nabla\pind,\nabla\frac{\vert\H\vert^2}{\pind}\right\rangle\nonumber\\
&=&-\frac{2}{\pind}\vert\nabla^\perp\H\vert^2+\frac{1}{2}\,\frac{\vert\H\vert^2\vert\nabla\pind\vert^2}{\pind^3}+\frac{2}{\pind}\left\langle\nabla\pind,\nabla\frac{\vert\H\vert^2}{\pind}\right\rangle\nonumber\\
&\le&\frac{1}{\pind}\left\langle\nabla\pind,\nabla\frac{\vert\H\vert^2}{\pind}\right\rangle,\nonumber
\end{eqnarray}
where the last step follows from
\begin{equation}\nonumber
\frac{2}{\pind}\vert\nabla^\perp\H\vert^2-\frac{1}{2}\,\frac{\vert\H\vert^2\vert\nabla\pind\vert^2}{\pind^3}-\frac{1}{\pind}\left\langle\nabla\pind,\nabla\frac{\vert\H\vert^2}{\pind}\right\rangle\ge 0,
\end{equation}
because
\begin{eqnarray}
0\le\left\vert\nabla\frac{\vert\H\vert^2}{\pind}\right\vert^2&=&\frac{\vert\nabla\vert \H\vert^2\vert^2-4\vert\H\vert^2\vert\nabla^\perp\H\vert^2}{\pind^2}\nonumber\\
&&+2\frac{\vert\H\vert^2}{\pind}\left(\frac{2}{\pind}\vert\nabla^\perp\H\vert^2-\frac{1}{2}\,\frac{\vert\H\vert^2\vert\nabla\pind\vert^2}{\pind^3}-\frac{1}{\pind}\left\langle\nabla\pind,\nabla\frac{\vert\H\vert^2}{\pind}\right\rangle\right)\nonumber\\
&\overset{(\ref{mean est})}{\le}&2\frac{\vert\H\vert^2}{\pind}\left(\frac{2}{\pind}\vert\nabla^\perp\H\vert^2-\frac{1}{2}\,\frac{\vert\H\vert^2\vert\nabla\pind\vert^2}{\pind^3}-\frac{1}{\pind}\left\langle\nabla\pind,\nabla\frac{\vert\H\vert^2}{\pind}\right\rangle\right),\nonumber
\end{eqnarray}
and $\nabla(\vert\H\vert^2\pind^{-1})$ vanishes if $\H=0$.
\end{proof}
\begin{corollary}\label{coro huni}
	Let $\M_t$ be a smooth solution of \eqref{mcf} with bounded curvature on each $\M_t$, and suppose $\M_0$ is the graph of a uniformly area decreasing map. Then
	$$\sup_{\M_t}\frac{\vert\H\vert^2}{\pind}\le \sup_{\M_0}\frac{\vert\H\vert^2}{\pind}.$$
\end{corollary}
\begin{proof}
	It follows from Corollary \ref{area uni} that $\M_t$ stays uniformly area decreasing and
	$$\inf_{\M_t}\pind\ge\inf_{\M_0}\pind>0.$$
	Since $\tind_{11}^2,\tind_{22}^2\le 1$ one has
	$$\vert\nabla\pind\vert^2=\sum_{k=1}^m\bigl(A^\xi_{1k}\tind_{11}+A^\eta_{2k}\tind_{22}\bigr)^2\le 2\vert A\vert^2.$$
	Therefore the gradient $\nabla\ln\pind$ is bounded on each finite time interval and we may apply Proposition \ref{prop 1} to the inequality in Lemma \ref{lemma 5.1} to conclude the result.
\end{proof}
\begin{theorem}[Longtime solution]\label{lte}
	Suppose $\M_0$ is the graph of a smooth and uniformly area decreasing map $\f:\real{m}\to\real{2}$ with $\sup_{\M_0}\vert\nabla^{\ell} A\vert$ bounded for all $\ell\ge 0$. Then there exists a unique smooth solution of \eqref{mcf} for all time such that each $\M_t$ is the graph of a uniformly area decreasing map satisfying \eqref{uniform} with the same constant $\varepsilon>0$, and 
	$$\sup_{\M_t}\vert\nabla^{\ell} A\vert\le C(\ell), \,\text{ for all } \ell\ge 0, t\in[0,\infty),$$
	where for each $\ell>0$ the constants $C(\ell)$ depend only on $\varepsilon$, $C(0)$, and $\sup_{\M_0}\vert\nabla^{j} A\vert$, $1\le j\le \ell$.
\end{theorem}
\begin{proof}
	The existence of a smooth solution on a maximal interval $[0,\tmax)$, $0<\tmax\le\infty$, follows from Proposition \ref{ste}. By the same proposition we conclude that there exists a maximal time $\tmax_{\max}$, $0<\tmax_{\max}\le\tmax$, such that $\M_t$ has bounded curvature on each $\M_t$, for all $t\in[0,\tmax_{\max})$. Then Proposition \ref{area uni} implies that
	$\inf_{\M_t}\pind\ge\varepsilon:=\inf_{\M_0}\pind>0$
	for all $t\in[0,\tmax_{\max})$, and from Corollary \ref{coro huni} follows that the mean curvature $\H$ is uniformly bounded on $[0,\tmax_{\max})$. We claim that $\vert A\vert$ is uniformly bounded on $[0,\tmax_{\max})$ as well. We argue by contradiction. Suppose $\vert A\vert$ is not uniformly bounded on $[0,\tmax_{\max})$. By definition of $\tmax_{\max}$ we have
	$$\lambda_k:=\sup_{\M\times[0,t_k]}\vert A\vert<\infty,\,\text{ for all }t_k<\tmax_{\max}.$$
	Thus there exists a sequence $\{(x_k,t_k)\}_{k\in\natural{}}\subset\M\times[0,\tmax_{\max})$ such that
	$$\vert A(x_k,t_k)\vert\ge \frac{\lambda_k}{2},\quad \lim_{k\to\infty}t_k=\tmax_{\max}.$$
	We now employ the same scaling as in \cite{CCH12} and define for each $k$ the parabolic rescaling at $(x_k,t_k)$ by\footnote{Note that here we have $\D u=\f$}
	$$\f_{\lambda_k}(y,s):=\lambda_k\Bigl(\f(x,t)-\f(x_k,t_k)\Bigr), \quad y:=\lambda_k(x-x_k),\quad s:=\lambda_k^2(t-t_k).$$
	Then for each fixed $k$, the graph of $\f_{\lambda_k}$ is again a solution of the graphical mean curvature flow on 
	the time interval $[-\lambda_k^2t_k,0]$ such that the rescaled second fundamental form $A_{\lambda_k}$ is 
	uniformly bounded from above by $1$, and $\vert A_{\lambda_k}(0,0)\vert\ge 1/2$, for all $k$. It follows from 
	Proposition \ref{longtime} and from the fact that $\f_{\lambda_k}$ is uniformly area decreasing with the same 
	constant $\varepsilon>0$ that $\sup_{\M_s}\vert\D\f_{\lambda_k}\vert$,
	$\sup_{\M_s}\vert\nabla^{\ell} A_{\lambda_k}\vert$ are uniformly bounded for all $\ell\ge 0$ and all $k$,
	$s\in[-\lambda_k^2t_k,0]$.  Hence by the \textsc{Arzela-Ascoli} theorem a subsequence of $\f_{\lambda_k}$ 
	converges locally uniformly on compact sets to a smooth limiting map $\f_{\infty}:\real{m}\to\real{2}$ with
	$|A_{\infty}(0)|\ge 1/2$.
	Since the rescaled mean curvature is given by
	$\H_{\lambda_k}(y,s)=\lambda_k^{-1}\H(x,t)$ 
	and $\H(x,t)$ is uniformly bounded it follows from $\lambda_k\to\infty$ that the mean curvature of the limit map
	$\f_{\infty}$ vanishes. Therefore $\f_{\infty}$ is a minimal and uniformly area decreasing map that is not totally 
	geodesic. This contradicts the \textsc{Bernstein} theorem in
	\cite[Theorem 1.1]{AJ18} 
	and proves that $A$ is indeed uniformly 
	bounded on $[0,\tmax_{\max})$. From Proposition \ref{longtime} and Proposition \ref{ste} we then conclude
	$\tmax_{\max}=\tmax=\infty$, and that for each $\ell\ge 0$ there exists a constant $C(\ell)$ such that
	$\sup_{\M_t}\vert\nabla^\ell A\vert\le C(\ell)$ holds uniformly for all $t\in[0,\infty)$. For $\ell>0$ the constants depend 
	on
	$\sup_{\M_0}\vert\nabla^j A\vert$, $1\le j\le \ell$ and on $\sup_{\M\times[0,\infty)}\vert A\vert$. This proves the 
	theorem.
\end{proof}

\section{Asymptotic convergence}\label{sec6}
In this section we address the behavior of longtime solutions of \eqref{mcf} for uniformly area decreasing maps. As we have seen in the previous section, longtime existence is guaranteed if the initial map $\f_0$ is uniformly area decreasing and all $C^\ell$-norms of $\f_0$, $\ell\ge 1$, are uniformly bounded. To describe the longtime behavior we need interior estimates. 
\begin{corollary}\label{mean int}
	Let $\M_t$ be a smooth solution of \eqref{mcf} with bounded curvature on each $\M_t$, and suppose $\M_0$ satisfies \eqref{uniform}. Then
	$$\sup_{\M_t}\frac{2t\vert\H\vert^2+m}{\pind}\le \frac{m}{\varepsilon}.$$
\end{corollary}
\begin{proof}
Recall that $\vert\H\vert^2\le m\vert A\vert^2$. From Lemma \ref{evolpind} and Lemma \ref{lemma 5.1}, we get
\begin{eqnarray}\label{mcf 7}
	\left(\dt-\Delta\right)\frac{2t\vert\H\vert^2+m}{\pind}
	&=&2t\left(\dt-\Delta\right)\frac{\vert\H\vert^2}{\pind}
	+2\frac{\vert\H\vert^2}{\pind}-\frac{m}{\pind^2}\left(\dt-\Delta\right)\pind-2m\frac{\vert\nabla\pind\vert^2}{\pind^3}\nonumber\\
	&\le&\frac{2t}{\pind}\left\langle\nabla\pind,\nabla\frac{\vert\H\vert^2}{\pind}\right\rangle+2\frac{\vert\H\vert^2-m\vert A\vert^2}{\pind}-\frac{3m}{2}\frac{\vert\nabla\pind\vert^2}{\pind^3}\nonumber\\
	&=&\frac{1}{\pind}\left\langle\nabla\pind,\nabla\frac{2t\vert\H\vert^2+m}{\pind}\right\rangle+2\frac{\vert\H\vert^2-m\vert A\vert^2}{\pind}-\frac{m}{2}\frac{\vert\nabla\pind\vert^2}{\pind^3}\nonumber\\
	&\le&\frac{1}{\pind}\left\langle\nabla\pind,\nabla\frac{2t\vert\H\vert^2+m}{\pind}\right\rangle.\nonumber
	\end{eqnarray}
	Since \eqref{uniform} holds for all $t\ge 0$ we can apply Proposition \ref{prop 1} to conclude
	$$\sup_{\M_t}\frac{2t\vert\H\vert^2+m}{\pind}\le\sup_{\M_0}\frac{m}{\pind}\le\frac{m}{\varepsilon},$$
	for all $t\ge 0$.
\end{proof}
The following counterexample demonstrates that this interior estimate is sharp in the sense that it does not necessarily  hold for area decreasing maps that are not uniformly area decreasing.
\begin{example}
	Let $\textbf{u}:\real{m}\to\real{}$, $m\ge 2$, be the bowl soliton, i.e. the graph of $\textbf{u}$ is the translating soliton in direction of $v:=e_{m+1}$ that is rotationally symmetric about the $e_{m+1}$-axis. Define the map $\f:\real{m}\to\real{2}$, given by
	$$\f(x):=(\textbf{u}(x),0).$$ Then $\f$ admits only one non-trivial singular value $\lambda$ and
	$\pind=(1+\lambda^2)^{-1}$ tends to zero for $\vert x\vert\to\infty$. Hence $\f$ is area decreasing but not uniformly area decreasing. Since the mean curvature of the translating soliton is time-independent, the interior mean curvature estimate in Corollary \ref{mean int} cannot hold in this case.
\end{example}
To derive explicit interior estimates for the second fundamental form $A$ is more difficult. To achieve this, we may apply the curvature estimate in \eqref{secest}. Unfortunately the coefficient in front of $\vert A\vert^4$ on the right hand side of that inequality is $3$ and not $2$, like in codimension one. We will show in the appendix that this imposes a constraint on the size of $\varepsilon$ in inequality \eqref{uniform}. Instead, we will now prove with a different method that the interior estimates for $\vert\nabla^\ell A\vert$ hold for any $\varepsilon>0$, and for arbitrary $\ell\ge 0$.
\begin{proposition}\label{prop int}
	Suppose $\M_0$ satisfies \eqref{uniform}, and $\sup_{\M_0}\vert\nabla^{\ell} A\vert$ is bounded for all $\ell\ge 0$. Then for each $\ell$ we obtain an interior estimate
	$$t^{\ell+1}\vert\nabla^{\ell} A\vert^2\le D(\ell),\, t\in[0,\infty),$$
	where the constants $D(\ell)$ depend only on $\varepsilon$, $D(0)$, and $\sup_{\M_0}\vert\nabla^{j} A\vert$, for every $1\le j\le \ell$.
\end{proposition}
\begin{proof}
	Longtime existence is guaranteed by Theorem \ref{lte}. It is well known that the interior estimates for
	$\vert\nabla^\ell A\vert$, $\ell\ge 1$, follow from those for $\vert A\vert$ by a standard bootstrapping argument similar 
	to that used in the proof of Proposition \ref{longtime} (cf. also with \cite{CCH12}). Therefore, we only need to prove 
	that $t\vert A\vert^2$ stays uniformly bounded. We argue by contradiction. Suppose $t\vert A\vert^2$ is not uniformly 
	bounded on $[0,\infty)$, and define
	$$\lambda_k:=\sup_{\M\times[0,t_k]}\vert A\vert.$$
	Thus there exists a sequence $\{(x_k,t_k)\}_{k\in\natural{}}\subset\M\times[0,\infty)$ such that
	$$\vert A(x_k,t_k)\vert\ge \frac{\lambda_k}{2},\quad \lim_{k\to\infty}t_k=\infty,\quad \lim_{k\to\infty}(t_k\lambda_k^2)=\infty.$$
	We now proceed exactly as in the proof of Theorem \ref{lte} and define for each $k$ the parabolic rescaling at $(x_k,t_k)$ by
	$$\f_{\lambda_k}(y,s):=\lambda_k\Bigl(\f(x,t)-\f(x_k,t_k)\Bigr), \quad y:=\lambda_k(x-x_k),\quad s:=\lambda_k^2(t-t_k).$$
	As in the proof of Theorem \ref{lte} we then conclude that a subsequence of $\f_{\lambda_k}$ converges locally 
	uniformly on compact sets to a smooth and uniformly area decreasing limiting map $\f_{\infty}:\real{m}\to\real{2}$ 
	with
	$$|A_{\infty}(0)|\ge 1/2.$$
	For each $k$ the rescaled mean curvature at $s=0$ is given by
	$$\H_{\lambda_k}(y,0)=\lambda_k^{-1}\,\H(x,t_k)=(\sqrt{t_k}\lambda_k)^{-1}\sqrt{t_k}\H(x,t_k),$$
	and $\sqrt{t_k}\,\H(x,t_k)$ is uniformly bounded.  It follows from $t_k\lambda_k^2\to\infty$ that the mean curvature of the limit 
	map $\f_{\infty}$ vanishes. Therefore $\f_{\infty}$ is a minimal and uniformly area decreasing map that is not totally 
	geodesic. This again contradicts the \textsc{Bernstein} theorem in
	\cite[Theorem 1.1]{AJ18} 
	and proves that $t|A|^2$ must be  uniformly bounded on $[0,\infty)$. 
\end{proof}
Similarly, as in \cite{EH}, we define
$$\tilde \F(x,s)=\frac{1}{\sqrt{2t+1}}\,\F(x,t),\quad s=\frac{1}{2}\log(2t+1),\,0\le s<\infty.$$
	The normalized evolution equation then becomes
\begin{equation}\label{mcfnorm}
\ds\tilde \F=\tilde{\H}-\tilde{\F},
\end{equation}
and the interior estimates in Proposition \ref{prop int} imply that on the rescaled submanifolds $\tilde \M_s:=\tilde \F(\M,s)$ we have
$$\tilde{\pind}\ge\varepsilon,\quad\vert\nabla^\ell\tilde A\vert^2\le C(\ell)$$
uniformly on $\tilde{\M}_s$, $s\in[0,\infty)$, for constants $C(\ell)$, depending only on the initial conditions $\varepsilon$, and on $C(0)$.

\begin{theorem}[Asymptotic convergence]\label{asymp}
	Suppose $\M_0$ is the graph of a uniformly area decreasing map with $\sup_{\M_0}\vert\nabla^{\ell} A\vert$ bounded for all $\ell\ge 0$, and suppose $\M_0$ is asymptotically conical in the sense of \eqref{cond cone}. Then the solution $\tilde{\M_s}$ of the normalized equation \eqref{mcfnorm} converges for $s\to\infty$ to a limiting submanifold $\tilde{\M}_\infty$ which is the graph of a uniformly area decreasing map, and in addition a self-expanding soliton, i.e. a solution of
	\begin{equation}\label{expander}
	\tilde \F^\perp=\tilde{\H}.
	\end{equation}
	In particular, there exist non-trivial uniformly area decreasing self-expanding smooth
	maps $\,\f:\real{m}\to\real{2}$.
\end{theorem}
\begin{remark}
	From \cite{EH} we get an interesting counterexample to Theorem \ref{asymp} in case \eqref{cond cone} is violated. Let $\f:\real{m}\to\real{2}$, $\f(x)=(\textbf{u} (x),0)$, where $\textbf{u}  :\real{m}\to\real{}$ is the function
	$$\textbf{u}  (x)=\begin{cases}
	\vert x\vert\sin\log\vert x\vert,&\vert x\vert \ge 1,\\
	\text{smooth},&\vert x\vert \le 1.
	\end{cases}$$
	$\f$ admits only one non-trivial singular value $\lambda$, and
	$$\pind=\frac{1}{1+\lambda^2}\ge\varepsilon>0$$
	for some constant $\varepsilon$ so that $\f$ is uniformly area decreasing. As was pointed out in \cite{EH}, the graph of $\textbf{u}$ (and hence also of $\f$) violates \eqref{cond cone} and does not asymptotically converge to a self-expander.
\end{remark}
Before we prove Theorem \ref{asymp} we need to compute certain evolution equations.
\begin{lemma}\label{lemma 6.6}
	The normalized quantity $\tilde{\H}-\tilde{\F}^\perp$ satisfies the evolution equation
	\begin{eqnarray*}
	\Bigl(\tilde{\nabla}^\perp_{\ds}-\tilde\Delta^\perp\Bigr)(\tilde{\H}-\tilde{\F}^\perp)&=&\langle \tilde A^{\tilde\H-\tilde\F^\perp},\tilde A\rangle-(\tilde\H-\tilde\F^\perp),\\
	\Bigl(\ds-\tilde\Delta\Bigr)\vert\tilde{\H}-\tilde\F^\perp\vert^2&=&-2\vert\tilde\nabla^\perp(\tilde\H-\tilde\F^\perp)\vert^2+2\vert \tilde A^{\tilde\H-\tilde\F^\perp}\vert^2-2\vert \tilde\H-\tilde\F^\perp\vert^2\\
	&\le&-2\vert\tilde\nabla^\perp(\tilde\H-\tilde\F^\perp)\vert^2+2(\vert \tilde A
	\vert^2-1)\vert \tilde\H-\tilde\F^\perp\vert^2.
	\end{eqnarray*}
\end{lemma}
\begin{proof}
	Similarly as in Lemma \ref{lemma 2.2} we get
	\begin{eqnarray*}
	\Bigl(\tilde\nabla^\perp_\ds-\tilde\Delta^\perp\Bigr)\tilde\H=\langle \tilde A^{\tilde\H},\tilde A\rangle+\tilde\H,\quad
	\Bigl(\tilde\nabla^\perp_\ds-\tilde\Delta^\perp\Bigr)\tilde \F^\perp=2\tilde \H+\langle \tilde A^{\tilde \F^\perp}\!,\tilde A\rangle-\tilde\F^\perp,
	\end{eqnarray*}
	and the result follows from the general inequality
	$$\vert\tilde A^\xi\vert^2\le \vert\tilde A\vert^2\vert\xi\vert^2,$$
	for any normal vector $\xi$.
\end{proof}
{\em Proof of Theorem $6.4$.}
	The proof can be carried out in a similar way as the proof of the asymptotic convergence result in \cite{EH}. Here the crucial function is
	$$\frac{\vert\H-\F^\perp\vert^2}{\pind(1+\alpha\vert\F\vert^2)^{1-\beta}},$$
	with constants $0<\beta<\delta$, $\alpha>0$ to be chosen, where $\delta$ denotes the constant in \eqref{cond cone}. Since $\pind$ is scaling invariant we obtain from \eqref{p11} that
	\begin{equation}\nonumber
	\Bigl(\ds-\Delta\Bigr)\tilde\pind\ge 2\tilde\pind\vert \tilde A\vert^2-\frac{1}{2\tilde\pind}\vert\nabla\tilde\pind\vert^2.
	\end{equation}
	Then with Lemma \ref{lemma 6.6} and with the inequality
	\begin{equation}\nonumber
	\vert\tilde\nabla\vert\tilde\H-\tilde\F^\perp\vert^2\vert^2\le 4\vert\tilde\H-\tilde \F^\perp\vert^2\vert\tilde\nabla^\perp(\tilde\H-\tilde\F^\perp)\vert^2
	\end{equation}
	we may proceed exactly as in the proof of Lemma \ref{lemma 5.1} to obtain
	\begin{equation*}
	\Bigl(\ds-\Delta\Bigr)\frac{\vert\tilde\H-\tilde \F^\perp\vert^2}{\tilde\pind}\le\frac{1}{\tilde\pind}\Bigl\langle\tilde\nabla\tilde\pind,\nabla\frac{\vert\tilde\H-\tilde \F^\perp\vert^2}{\tilde\pind}\Bigr\rangle-2\frac{\vert\tilde\H-\tilde \F^\perp\vert^2}{\tilde\pind}.
	\end{equation*}
This completes the proof. $\hfill \square $

\section*{Appendix}\label{app}
Here we will derive an explicit estimate for the second fundamental form of
the graph that is interior in time, provided the initial map is 
uniformly area decreasing with a sufficiently big constant $\varepsilon$. Thus in this case we obtain an alternative proof of 
the interior estimates given in Proposition \ref{prop int}, and in particular the constant $C(0)$ in Proposition \ref{prop int} 
will then only depend on $\sup_{\M_0}\vert A\vert$.

\smallskip
Let $\varmine$ be a positive non-decreasing function of $\pind$ to be determined later and set $\varmine'=\partial\varmine/\partial{\pind}$. For a time dependent function $\alpha\ge0$ we define
$$w:=\frac{\alpha\vert A\vert^2+1}{\varmine}.$$
We compute
\begin{eqnarray}
\left(\dt-\Delta\right)w&=&\frac{1}{\varmine}\left(\dt-\Delta\right)\bigl(\alpha\vert A\vert^2+1\bigr)-w\frac{\varmine'}{\varmine}\left(\dt-\Delta\right)\pind\nonumber\\
&&+w\Bigl(\frac{\varmine''}{\varmine}-2\Bigl(\frac{\varmine'}{\varmine}\Bigr)^2\Bigr)\vert\nabla\pind\vert^2+2\frac{\varmine'}{\varmine^2}\langle\nabla\pind,\nabla(\alpha\vert A\vert^2)\rangle\nonumber\\
&\overset{(\ref{secest}),(\ref{p11})}{\le}&\frac{\vert A\vert^2}{\varmine}\bigl(\bdot\alpha+3\alpha\vert A\vert^2-2\pind\varmine' w\bigr)
+2\frac{\varmine'}{\varmine}\langle\nabla\pind,\nabla w\rangle\nonumber\\
&&-2\frac{\alpha}{\varmine}\vert\nabla^\perp A\vert^2+\Bigl(\frac{\varmine''}{\varmine}+\frac{\varmine'}{2\pind\varmine}\Bigr)w\vert\nabla\pind\vert^2.\label{inter 3}
\end{eqnarray}
With \textsc{Kato}'s inequality $\vert\nabla\vert  A\vert^2\vert^2\le 4\vert A\vert^2\vert\nabla^\perp A\vert^2$ one gets
\begin{eqnarray}
0\le\vert\nabla w\vert^2&=&\left\vert\frac{\alpha}{\varmine}\nabla\vert A\vert^2-\frac{\varmine'}{\varmine} w\nabla\pind\right\vert^2\nonumber\\
&=&\frac{\alpha^2}{\varmine^2}\vert\nabla\vert A\vert^2\vert^2-2\frac{\alpha\varmine'}{\varmine^2} w\langle\nabla\vert A\vert^2,\nabla \pind\rangle+\Bigl(\frac{\varmine'}{\varmine}\Bigr)^2w^2\vert\nabla\pind\vert^2\nonumber\\
&\le&2 w\left(\frac{2\alpha}{\varmine}\vert\nabla^\perp A\vert^2-\frac{\varmine'}{\varmine} \langle\nabla w,\nabla\pind\rangle-\frac{1}{2}\Bigl(\frac{\varmine'}{\varmine}\Bigr)^2w\vert\nabla\pind\vert^2\right).\nonumber
\end{eqnarray}
Hence $\alpha\ge 0$ implies
$$-2\frac{\alpha}{\varmine}\vert\nabla^\perp A\vert^2\le-\frac{\varmine'}{\varmine} \langle\nabla w,\nabla\pind\rangle-\frac{1}{2}\Bigl(\frac{\varmine'}{\varmine}\Bigr)^2w\vert\nabla\pind\vert^2.$$
Inserting this into \eqref{inter 3} gives
\begin{eqnarray}
\left(\dt-\Delta\right)w
&\le&\frac{\varmine'}{\varmine}\langle\nabla\pind,\nabla w\rangle+\frac{\vert A\vert^2}{\varmine}\bigl(\bdot\alpha+3\alpha\vert A\vert^2-2\pind\varmine' w\bigr)\nonumber\\
&&+\left(2\pind\frac{\varmine''}{\varmine}+\frac{\varmine'}{\varmine}-\pind\Bigl(\frac{\varmine'}{\varmine}\Bigr)^2\right)\frac{w\vert\nabla\pind\vert^2}{2\pind}.\label{sec prelim}
\end{eqnarray}
\begin{corollary}\label{coro 7.1}
	Let $\M_t$ be a smooth solution of \eqref{mcf} with bounded curvature on each $\M_t$, and suppose $\M_0$ satisfies \eqref{uniform} with $\varepsilon>1/9$. Then we have the interior estimate
	$$t\vert A\vert^2\le \frac{2}{(3\sqrt{\varepsilon}-1)^2}.$$
\end{corollary}
\begin{proof}
We use \eqref{sec prelim} with $\alpha(t)=3t$, and
$$\varmine(\pind)=(\sqrt{\pind}-\sqrt{\varrho})^2,\quad \text{with} \quad1/9<\varrho<\varepsilon.$$ Since $\pind\ge\varepsilon$ is preserved, we get $\varmine>0$ for all $t$. Moreover, $\varmine$ is the solution of the ODE
$$2\pind\frac{\varmine''}{\varmine}+\frac{\varmine'}{\varmine}-\pind\Bigl(\frac{\varmine'}{\varmine}\Bigr)^2=0,\quad\varmine(\varrho)=0.$$
So with this choice of $\varmine$ inequality \eqref{sec prelim} becomes
\begin{eqnarray}
\left(\dt-\Delta\right)w
&\le&\frac{\varmine'}{\varmine}\langle\nabla\pind,\nabla w\rangle+\frac{\vert A\vert^2}{\varmine}\bigl(\bdot\alpha+3\alpha\vert A\vert^2-2\pind\varmine' w\bigr)\nonumber\\
&=&\frac{\varmine'}{\varmine}\langle\nabla\pind,\nabla w\rangle+\Bigl(3-2\pind\frac{\varmine'}{\varmine}\Bigr)\vert A\vert^2w.\nonumber
\end{eqnarray}
On the other hand $\varrho>1/9$, and $\pind< 1$ implies
$$3\varmine-2\pind\varmine' =3(\sqrt{\pind}-\sqrt{\varrho})^2-2\sqrt{\pind}(\sqrt{\pind}-\sqrt{\varrho})=(\sqrt{\pind}-3\sqrt{\varrho})(\sqrt{\pind}-\sqrt{\varrho})<0$$
which gives
\begin{eqnarray}
\left(\dt-\Delta\right)w
&\le&\frac{\varmine'}{\varmine}\langle\nabla\pind,\nabla w\rangle.
\end{eqnarray}
The maximum principle in Proposition \ref{prop 1} gives
$$\frac{3t\vert A\vert^2+1}{(\sqrt{\pind}-\sqrt{\varrho})^2}\le\sup_{\M_0}\frac{1}{(\sqrt{\pind}-\sqrt{\varrho})^2}\le\frac{1}{(\sqrt{\varepsilon}-\sqrt{\varrho})^2},$$
for any $\varrho\in(1/9,\varepsilon)$. For $\varrho\to1/9$ we thus obtain, using $\pind\le1$,
$$3t\vert A\vert^2\le\frac{(\sqrt{\pind}-1/3)^2}{(\sqrt{\varepsilon}-1/3)^2}-1\le\frac{4}{(3\sqrt{\varepsilon}-1)^2},$$
and the claim follows.
\end{proof}
\begin{remark}
Under the assumption $\pind>1/9$, Corollary \ref{coro 7.1} gives an alternative proof of a weaker version of the \textsc{Bernstein} theorem in \cite{AJ18} by letting $t$ tend to infinity.
For general \textsc{Bernstein} type theorems concerning entire area decreasing graphs of arbitrary codimension, we refer to \cite{DJX21}.
\end{remark}


\begin{thebibliography}{99}

\bibitem{AJ18}
	R.\,Assimos, and J.\,Jost,
	\emph{The geometry of maximum principles and a Bernstein theorem in codimension 
	2}, (2018), 27 pp.

\bibitem{ASS22}
	R.\,Assimos, A.\,Savas-Halilaj, and K.\,Smoczyk,
	\emph{Graphical mean curvature flow with bounded bi-Ricci curvature},
	Calc. Var. \textbf{62:1}(2023), Paper No. 12, 26 pp.
	
	\bibitem{CCH12}
	A.\,Chau, J.\,Chen, and W.\,He, 
	\emph{Lagrangian mean curvature flow for entire Lipschitz graphs},
	Calc. Var.
	\textbf{44:1-2}(2012), 199--220.
	
\bibitem{CCY13}
	A.\,Chau, J.\,Chen, and Y.\,Yuan, 
	\emph{ Lagrangian mean curvature flow for entire Lipschitz graphs II.},
	Math. Ann.
	\textbf{357:1}(2013),  165--183.
	
\bibitem{DS23}
	P.\,Daskalopoulos, and M.\,S{\'a}ez, 
	\emph{Uniqueness of entire graphs evolving by mean curvature flow},
	J. Reine Angew. Math.
	\textbf{796}(2023),  201--227.
	
	\bibitem{DJX21}
	Q.\,Ding, J.\,Jost, and Y.-L.\,Xin,
	\emph{Minimal graphs of arbitrary codimension in Euclidean space with bounded
	$2$-dilation},
	(2021), 56 pp.

\bibitem{DJX23}
	Q.\,Ding, J.\,Jost, and Y.-L.\,Xin,
	\emph{Existence and non-existence of minimal graphs},
	J. Math. Pures Appl.
	\textbf{ 179:9}(2023), 391--424.

	
\bibitem{ecker}
	K.\,Ecker,
	\emph{Regularity theory for mean curvature flow},
	Prog. Nonlinear Differ. Equ. Appl.
	\textbf{57}(2004), 165 pp.

\bibitem{EH3}
	K.\,Ecker, and G.\,Huisken,
	\emph{Interior estimates for hypersurfaces moving by mean curvature},
	Invent. Math.
	\textbf{105:3}(1991), 547--569.


\bibitem{EH}
	K.\,Ecker, and G.\,Huisken,
	\emph{Mean curvature evolution of entire graphs},
	Ann. of Math. (2)
	\textbf{130:3}(1989), 453--471.
	
\bibitem{ES}
	L.\,C.\,Evans, and J.\,Spruck,
	\emph{Motion of level sets by mean curvature I.},
	J. Differential Geom.
	\textbf{33:3}(1991), 635--681.
	
	\bibitem{Giga}
	Y.\,Giga,
	\emph{Surface evolution equations - A level set approach},
	Monographs in Mathematics, Birkh{\"a}user Verlag, Basel, vol. \textbf{99}(2006),
	264 pp..
	
\bibitem{Hamilton82}
	R.\,S.\,Hamilton, 
	\emph{Three-manifolds with positive Ricci curvature},
	J. Differential Geom.,
	\textbf{17}(1982), 255--306.

\bibitem{Hamilton93}
	R.\,S.\,Hamilton, 
	\emph{Monotonicity formulas for parabolic flows on manifolds},
	Comm. Anal. Geom.,
	\textbf{1:1}(1993), 127--137.

\bibitem{Huisken84}
	G.\,Huisken,
	\emph{Flow by mean curvature of convex surfaces into spheres},
	J. Differential Geom.,
	\textbf{20}(1984), 237--266.

\bibitem{Huisken90}
	G.\,Huisken,
	\emph{Asymptotic behavior for singularities of the mean curvature flow},
	J. Differential Geom.,
	\textbf{31:1}(1990), 285--299.

\bibitem{LiLi92}
   A.-M.\, Li, and J.\,Li,
	\emph{An intrinsic rigidity theorem for minimal submanifolds in a sphere},
	Arch. Math. (Basel),
	\textbf{58:6}(1992), 582--594.
	
\bibitem{Lubbe1}
   F.\,Lubbe,
	\emph{Evolution of area-decreasing maps between two-dimensional Euclidean spaces},
	J. Geom. Anal.
	\textbf{28:4}(2018).
	
\bibitem{Lubbe2}
   F.\,Lubbe,
	\emph{Mean curvature flow of contractions between Euclidean spaces},
	Calc. Var.
	\textbf{55:4}(2016), Paper No. 104, 28 pp.

\bibitem{Saez}
	M.\,S{\'a}ez, and O.\,Schn{\"u}rer,
	\emph{Mean curvature flow without singularities},
	J. Differential Geom.,
	\textbf{97:3}(2014), 545--570.

\bibitem{SS14}
	A.\,Savas-Halilaj, and K.\,Smoczyk,
	\emph{Homotopy of area decreasing maps by mean curvature flow},
	Adv. Math.,
	\textbf{255}(2014), 455--473.

\bibitem{SS15}
	A.\,Savas-Halilaj, and K.\,Smoczyk,
	\emph{Evolution of contractions by mean curvature flow},
	Math. Ann.,
	\textbf{361:3-4}(2015), 725--740.

\bibitem{SS18}
	A.\,Savas-Halilaj, and K.\,Smoczyk,
	\emph{Mean curvature flow of area decreasing maps between Riemann surfaces},
	Ann. Global Anal. Geom.,
	\textbf{53}(2018), 11--37.

\bibitem{Smoczyk12}
	K.\,Smoczyk,
	\emph{Mean curvature flow in higher codimension - Introduction and survey}, 
	Global Differential Geometry,  Springer Proceedings in Mathematics,
	\textbf{12}(2012), 231--274.

\end{thebibliography}
\end{document}